\documentclass[11pt, a4paper]{article}

\usepackage{pb-diagram, lamsarrow,pb-lams}
\usepackage{amssymb}
\usepackage{amsthm}
\usepackage{amsmath}
\usepackage[T1]{fontenc}
\usepackage{latexsym}
\usepackage{textcomp}
\usepackage[ansinew]{inputenc}
\usepackage{exscale}
\usepackage{graphicx}
\usepackage{enumerate}

\input{amssym.def}
\input{amssym}
\input xy
\xyoption{all} \CompileMatrices

    \newcommand{\cl}{\mathrm{cl}}
    
    \newcommand{\pd}{\mathrm{pd}}

    \newcommand{\aug}{\mathrm{aug\,}}

    \newcommand{\irr}{\mathrm{Irr\,}}
    \newcommand{\irrp}{\mathrm{Irr_p\,}}
    \newcommand{\tor}{_{\mathrm{tor}}}

    \newcommand{\PMod}{\mathrm{PMod}}
    
    \newcommand{\perf}{\mathrm{perf}}
    \newcommand{\Dperf} {\mathcal D ^{\perf}}
    \newcommand{\ith}{^{\mathrm{th}}}
    
    \newcommand{\Ann}{\mathrm{Ann}}
    \newcommand{\Fitt}{\mathrm{Fitt}}

    \newcommand{\Spec}{\mathrm{Spec}}
    \newcommand{\nr}{\mathrm{nr}}
    \newcommand{\Gl}{\mathrm{Gl}}

    \newcommand{\et}{\mathrm{\acute{e}t}}
    \newcommand{\Div}{\mathrm{Div}}
    \newcommand{\cone}{\mathrm{cone}}

    \newcommand{\punkt}{^{\cdot}}

    \newcommand{\qpc} {\Q_p^{\mr{c}}}
    
    \newcommand{\zpg} {\zp G}
    \newcommand{\qpg} {\qp G}
    \newcommand{\cp} {\C_p}

    \newcommand{\arat} {\ar@}
    \newcommand{\inj} {\arat{^{(}->}}
    \newcommand{\sur} {\arat{->>}}
    \newcommand{\equal} {\arat{=}}
    \def\barat[#1] {\arat{->}'[#1][#1#1]}
    \def\binj[#1] {\inj'[#1][#1#1]}
    \def\bsur[#1] {\sur'[#1][#1#1]}
    \def\bequal[#1] {\equal'[#1][#1#1]}

    \newcommand{\ram}{_{\mathrm{ram}}}

    \newcommand{\kot}[1]{K_{0}T({#1})}

    \newcommand{\beq}{\begin{equation}}
    \newcommand{\eeq}{\end{equation}}

    \newcommand{\mc}{\mathcal}

    \newcommand{\half}{{\textstyle \frac{1}{2}}}

     \newcommand{\fa}{\goth{a}}
     \newcommand{\fo}{\goth{o}}
     \newcommand{\fp}{\goth{p}}
     \newcommand{\fA}{\goth{A}}
     \newcommand{\fM}{\goth{M}}
     \newcommand{\fP}{\goth{P}}

     \newcommand{\C}{\mathbb{C}}
    \newcommand{\N}{\mathbb{N}}
     \newcommand{\Q}{\mathbb{Q}}
     \newcommand{\R}{\mathbb{R}}
    \newcommand{\Z}{\mathbb{Z}}

    \newcommand{\ol}[1]{\overline{#1}}

    \newcommand{\ti}[1]{\tilde{#1}}
   \newcommand{\sm}{\setminus}
    \newcommand{\op}{\oplus}
    
    \newcommand{\me}{^{-1}}
    \newcommand{\mal}{^{\times}}
    \newcommand{\mr}{\mathrm}
    \newcommand{\clo}{^{\mr{c}}}

    \newcommand{\zg}{{\mathbb{Z}G}}

    \newcommand{\rg}{{\mathbb{R}G}}

    \newcommand{\zp}{{\mathbb{Z}_p}}
    
    \newcommand{\qp}{{\mathbb{Q}_p}}

    \newcommand{\into}{\rightarrowtail}
    \newcommand{\onto}{\twoheadrightarrow}
    \newcommand{\lto}{\longrightarrow}

    \newcommand{\ka}{\kappa}
    \newcommand{\ga}{\gamma}
    \newcommand{\Ga}{\Gamma}
    \newcommand{\si}{\sigma}
    
    \newcommand{\de}{\delta}
    \newcommand{\De}{\Delta}

    \newcommand{\La}{\Lambda}

    \newcommand{\om}{\omega}
    \newcommand{\al}{\alpha}
    \newcommand{\ve}{\varepsilon}

     \newcommand{\barr}{\begin{array}}
    \newcommand{\earr}{\end{array}}
    \newcommand{\bit}{\begin{itemize}}
    \newcommand{\eit}{\end{itemize}}
    \newcommand{\ben}{\begin{enumerate}}
    \newcommand{\een}{\end{enumerate}}
   \newcommand{\bea}{\begin{eqnarray*}}
    \newcommand{\eea}{\end{eqnarray*}}

    \newcommand{\Hom}{\mathrm{Hom}}

   \newcommand{\Gal}{\mathrm{Gal}}
    
    \newcommand{\ord}{\mathrm{ord\,}}
    
   \newcommand{\ind}{\mathrm{ind\,}}

   \newcommand{\res}{\mathrm{res\,}}
   \newcommand{\Det}{\mathrm{Det\,}}

\newtheorem{theo}{Theorem}[section]
\newtheorem{prop}[theo]{Proposition}
\newtheorem{lem}[theo]{Lemma}
\newtheorem{cor}[theo]{Corollary}
\newtheorem{defi}[theo]{Definition}
\newtheorem{con}[theo]{Conjecture}
\newtheorem{rem}[theo]{Remark}

\newcommand{\myfootnote}[1]{%
\renewcommand{\thefootnote}{}%
\footnotetext{#1}%
\renewcommand{\thefootnote}{\arabic{footnote}}%
}

\begin{document}
\xymatrixrowsep{3pc} \xymatrixcolsep{3pc}

\title{Equivariant Iwasawa theory and non-abelian Stark-type conjectures}
\author{Andreas Nickel\thanks{I acknowledge financial support provided by the DFG}}
\date{}

\maketitle

\begin{abstract}
    We discuss three different formulations of the equivariant Iwasawa main conjecture attached to
    an extension $\mc K/k$ of totally real fields with Galois group $\mc G$,
    where $k$ is a number field and  $\mc G$ is a $p$-adic Lie group of dimension $1$ for an odd prime $p$.
    All these formulations are equivalent and hold if Iwasawa's $\mu$-invariant vanishes. Under mild hypotheses, we use this to
    prove non-abelian generalizations of Brumer's conjecture, the Brumer-Stark conjecture and a strong
    version of the Coates-Sinnott conjecture provided that $\mu = 0$.
\end{abstract}

\section*{Introduction}
\myfootnote{{\it 2010 Mathematics Subject Classification:} 11R23, 11R42}
\myfootnote{{\it Keywords:} Iwasawa theory, main conjecture, equivariant $L$-values, Stark conjectures}
Let $K/k$ be a finite Galois CM-extension of number fields with Galois group $G$. To each finite set $S$ of places of
$k$ which contains all the infinite places, one can associate a so-called ``Stickelberger element''
$\theta_S(K/k)$ in the center of the group ring algebra $\C G$. This Stickelberger element is defined via
$L$-values at zero of $S$-truncated Artin $L$-functions attached to the (complex) characters of $G$.
Let us denote the roots of unity of $K$ by $\mu_K$ and the class group of $K$ by $\cl_K$.
Assume that $S$ contains the set $S_{\ram}$ of all finite primes of $k$ which ramify in $K/k$. Then it was independently
shown in \cite{Cassou}, \cite{Deligne-Ribet} and \cite{Barsky} that for abelian $G$ one has
\beq \label{Ca-DR-Ba}
    \Ann_{\zg} (\mu_K) \theta_S(K/k) \subset \zg.
\eeq
Now Brumer's conjecture asserts that $\Ann_{\zg} (\mu_K) \theta_S(K/k)$ annihilates $\cl_K$.
There is a large body of evidence in support of Brumer's conjecture (cf.~the expository article \cite{Brumer-expository});
in particular, Greither \cite{Gr-Fitt-ETNC} has shown that the appropriate special case of the
equivariant Tamagawa number conjecture (ETNC) implies the $p$-part of Brumer's conjecture for an odd prime $p$
if the $p$-part of $\mu_K$ is a c.t.~(short for cohomologically trivial) $G$-module. A similar result for arbitrary
$G$ was proven by the author \cite{ich-Fitting}, improving an unconditional annihilation result
due to Burns and Johnston \cite{Burns_Johnston}. Note that the assumptions made in loc.cit.~are
adapted to ensure the validity of the strong Stark conjecture.
Moreover, in \cite{ich-stark}, the author has introduced non-abelian generalizations of Brumer's conjecture, the Brumer-Stark conjecture
and of the so-called strong Brumer-Stark property.
The extension $K/k$ fulfills the latter if certain Stickelberger elements are contained in
the (non-commutative) Fitting invariants of corresponding ray class groups; but it
does not hold in general, even if $G$ is abelian,
as follows from the results in \cite{Gr-Kurihara}. But if this property happens to be true, this also implies the
validity of the (non-abelian) Brumer-Stark conjecture and Brumer's conjecture. We will prove the $p$-part of a dual version
of the strong Brumer-Stark property for an arbitrary CM-extension of number fields and an odd prime $p$ under the only restriction
that $S$ contains all the $p$-adic places of $k$ and that Iwasawa's $\mu$-invariant vanishes. In particular,
this implies the (non-abelian) Brumer-Stark conjecture and Brumer's conjecture under the same hypotheses.
Note that the vanishing of $\mu$ is a long standing
conjecture of Iwasawa theory; the most general result is still due to Ferrero
and Washington \cite{Ferrero_Wash} and says that $\mu=0$ for absolutely abelian extensions.\\

We have to discuss three different versions of the equivariant Iwasawa main conjecture (EIMC). The first formulation
is due to Ritter and Weiss \cite{towardII}, the second follows the framework of \cite{CFKSV} and was used by Kakde \cite{Kakde-mc}
in his proof of the EIMC. Finally, Greither and Popescu \cite{EIMC-reform} have formulated an EIMC
via the Tate module of a certain Iwasawa-theoretic abstract $1$-motive; but they restrict their formulation to abelian
extensions. So one of our first tasks is to give a formulation of their conjecture in the non-abelian situation as well.
In fact, it will be this formulation which will lead to the above mentioned proof of the (dual) strong Brumer-Stark property.
All variants of the EIMC hold if Iwasawa's $\mu$-invariant vanishes. This follows from the recent result
of Ritter and Weiss \cite{EIMC-theorem} on the EIMC for $p$-adic Lie groups of dimension $1$. In fact,
this can be generalized to Lie groups of higher dimension as shown by Kakde \cite{Kakde-mc} and, independently, by Burns \cite{Burns-mc}.
Note that Kakde in fact provides an independent proof also in the case of dimension $1$.\\

Finally, we will introduce a (non-abelian) analogue of the strong Brumer-Stark property for higher \'{e}tale cohomology groups.
For abelian extensions, this property implies the Coates-Sinnott conjecture, and for arbitrary extensions it implies
a non-abelian analogue of this conjecture which is closely related to the central conjecture in \cite{ich-negative}.
In contrast to the strong Brumer-Stark property, we conjecture that its higher analogue holds in general and
we consequently will call this conjecture the (non-abelian) strong Coates-Sinnott conjecture.
We provide several reduction steps which under certain mild hypotheses allows us to assume
that $K/k$ is a Galois CM-extension. In this situation, we show that the strong Coates-Sinnott conjecture
is (nearly) equivalent to an appropriate special case of the ETNC. We may conclude that the strong
Coates-Sinnott conjecture holds provided that $\mu=0$, since these special cases of the ETNC have been proven
by Burns \cite{Burns-mc} under this assumption. We also give a direct proof of our conjecture, still assuming
that $\mu=0$, using our new formulation of the EIMC. This will also provide a new proof
of Burns' result on the ETNC.\\

This article is organized as follows. In section $1$, we provide the necessary background material. In particular,
we discuss the notion of non-commutative Fitting invariants which have been introduced by the author \cite{ich-Fitting},
and how we may define them for certain perfect complexes. In section $2$, we give the formulation of the EIMC
due to Ritter and Weiss, but using Fitting invariants rather than the $\Hom$ description.
We show that the canonical complex which occurs in the construction of Ritter and Weiss is isomorphic in the derived
category of Iwasawa modules to $R\Hom(R\Ga_{\et}(\Spec(\fo_{\mc K}[\frac{1}{S}]), \qp / \zp), \qp / \zp)$.
This will explain the relation of the first two above mentioned formulations of the EIMC in more
detail than it is available in the literature so far. In section $3$, we recall the notion of
abstract $1$-motives as formulated in \cite{EIMC-reform} and show, how to use the Iwasawa-theoretic
abstract $1$-motive of \cite{EIMC-reform} to formulate an EIMC in the non-abelian situation as well.
Assuming the vanishing of $\mu$, we deduce this conjecture from the result on the EIMC due to
Ritter and Weiss \cite{EIMC-theorem}. In fact, the argument can be reversed such that both conjectures are equivalent.
In section $4$, we use our new formulation of the EIMC to prove the above mentioned cases
of the (dual) strong Brumer-Stark property. Finally, we introduce and discuss the strong Coates-Sinnott conjecture
in section $5$.\\

The author would like to thank Cornelius Greither for several discussions concerning the article \cite{EIMC-reform}.

\section{Preliminaries}

\subsubsection{$K$-theory}
Let $\La$ be a left noetherian ring with $1$ and $\PMod(\La)$ the
category of all finitely generated projective $\La$-modules. We
write $K_0(\La)$ for the Grothendieck group of $\PMod(\La)$, and
$K_1(\La)$ for the Whitehead group of $\La$ which is the abelianized
infinite general linear group. If $S$ is a multiplicatively closed
subset of the center of $\La$ which contains no zero divisors, $1
\in S$, $0 \not \in S$, we denote the Grothendieck group of the
category of all finitely generated $S$-torsion $\La$-modules of
finite projective dimension by $K_0S(\La)$.
Writing $\La_S$ for the
ring of quotients of $\La$ with denominators in $S$,
we have the following
Localization Sequence (cf.~\cite{CR-II}, p.~65)
\beq \label{localization_sequence}
    K_1(\La) \to K_1(\La_S) \stackrel{\partial}{\lto} K_0S(\La) \stackrel{\rho}{\lto} K_0(\La) \to K_0(\La_S).
\eeq
In the special case where $\La$ is an $\fo$-order over a commutative ring $\fo$ and $S$ is the set of all
nonzerodivisors of $\fo$, we also write $\kot\La$ instead of $K_0S(\La)$.
Moreover, we denote the relative $K$-group corresponding to
a ring homomorphism $\La \to \La'$ by $K_0(\La,\La')$ (cf.~\cite{Swan}).
Then we have a
Localization Sequence (cf.~\cite{CR-II}, p.~72)
\[
    K_1(\La) \to K_1(\La') \stackrel{\partial_{\La,\La'}}{\lto} K_0(\La,\La') \to K_0(\La) \to K_0(\La').
\]
It is also shown in \cite{Swan} that there is an isomorphism $K_0(\La,\La_S) \simeq K_0S(\La)$.
For any ring $\La$ we write $\zeta(\La)$ for the subring of all elements which are central in $\La$.
Let $L$ be a subfield of either $\C$ or $\cp$ for some prime $p$ and let $G$ be a finite group.
In the case where $\La'$ is the group ring $LG$ the reduced norm map $\nr_{LG}: K_1(LG) \to \zeta(LG)\mal$
is always injective. If in addition $L=\R$,
there exists a canonical map $\hat \partial_G: \zeta(\rg)\mal \to K_0(\zg, \rg)$ such that the restriction
of $\hat \partial_G$ to the image of the reduced norm equals $\partial_{\zg, \rg} \circ \nr_{\rg}\me$.
This map is called the extended boundary homomorphism and was introduced by Burns and Flach \cite{Burns_Flach}.\\

For any ring $\La$ we write $\mathcal D (\La)$ for the derived
category of $\La$-modules. Let $\mathcal C^b (\PMod (\La))$ be the
category of bounded complexes of finitely generated projective
$\La$-modules. A complex of $\La$-modules is called perfect if it is
isomorphic in $\mathcal D (\La)$ to an element of $\mathcal C^b
(\PMod (\La))$. We denote the full triangulated subcategory of
$\mathcal D (\La)$ consisting of perfect complexes by $\mathcal D
^{\perf} (\La)$. For any $C\punkt \in \mathcal C^b (\PMod (\La))$ we
define $\La$-modules
$$C^{ev} := \bigoplus_{i \in \Z} C^{2i},~ C^{odd} := \bigoplus_{i \in \Z} C^{2i+1}.$$
Similarly, we define $H^{ev}(C\punkt)$ and $H^{odd}(C\punkt)$ to be the direct sum over all even (resp.~odd) degree
cohomology groups of $C\punkt$.\\

For the following let $R$ be a Dedekind domain of characteristic
$0$, $K$ its field of fractions, $A$ a finite dimensional
$K$-algebra and $\La$ an $R$-order in $A$. A pair $(C\punkt,t)$
consisting of a complex $C\punkt \in \mathcal \Dperf (\La)$ and an
isomorphism $t: H^{odd}(C_K \punkt) \to H^{ev}(C_K\punkt)$ is called a
trivialized complex, where $C_K\punkt$ is the complex obtained by
tensoring $C\punkt$
with $K$. We refer to $t$ as a trivialization of $C\punkt$.
One defines the refined Euler characteristic $\chi_{\La,A}
(C\punkt, t) \in K_0(\La,A)$ of a trivialized complex as follows:
Choose a complex $P\punkt \in \mathcal C^b(\PMod(R))$ which is
quasi-isomorphic to $C\punkt$. Let $B^i(P_K \punkt)$ and $Z^i(P_K
\punkt)$ denote the $i\ith$ cobounderies and $i \ith$ cocycles of
$P_K \punkt$, respectively. We have the obvious exact sequences
$$ {B^i(P_K\punkt)} \into {Z^i(P_K\punkt)} \onto {H^i(P_K\punkt)}\mbox{~,~~ }
   {Z^i(P_K\punkt)} \into {P_K^i} \onto {B^{i+1}(P_K\punkt).}  $$
If we choose splittings of the above sequences, we get an
isomorphism
$$   \phi_t: P_K^{odd}  \simeq  \bigoplus_{i \in \Z} B^i(P_K\punkt) \op H^{odd}(P_K\punkt)
      \simeq  \bigoplus_{i \in \Z} B^i(P_K\punkt)  \op H^{ev}(P_K\punkt)
      \simeq  P_K^{ev},$$
where the second map is induced by $t$. Then the refined
Euler characteristic is defined to be
$$\chi_{\La, A} (C\punkt, t) := (P^{odd}, \phi_t, P^{ev}) \in K_0(\La, A)$$
which indeed is independent of all choices made in the
construction.
For further information concerning refined Euler characteristics
we refer the reader to \cite{Burns_Whitehead}.\\

Denote the full triangulated subcategory of
$\mathcal D (\La)$ consisting of perfect complexes whose cohomologies are $R$-torsion by
$\mathcal D^{\perf}\tor (\La)$. For any complex $C\punkt \in \mathcal D^{\perf}\tor (\La)$
there is a unique trivialization, namely $t = 0$; hence $C\punkt$ defines a class
$$[C\punkt] := \chi_{\La,A}(C\punkt, 0) \in K_0(\La,A) = \kot{\La}.$$
In fact, $K_0(\La, A)$ identifies with the Grothendieck group whose generators are $[C\punkt]$, where $C\punkt$
is an object of the category $\mathcal C^b\tor(\PMod(\La))$ of bounded complexes of finitely generated projective $\La$-modules whose cohomologies are
$R$-torsion, and the relations are as follows: $[C\punkt] = 0$ if $C\punkt$ is acyclic, and
$[C_2\punkt] = [C_1\punkt] + [C_3\punkt]$ for any short exact sequence
$$C_1\punkt \into C_2\punkt \onto C_3\punkt$$
in $\mathcal C^b\tor(\PMod(\La))$ (cf.~\cite{Weibel}).
Moreover, if $M$ is a finitely generated $R$-torsion $\La$-module of finite projective dimension, then the class
of $M$ in $\kot{\La}$ agrees with the class $[M] \in K_0(\La,A)$, where $M$ is considered as a perfect complex
concentrated in degree $1$.

\subsubsection{Non-commutative Fitting invariants}
For the following we refer the reader to \cite{ich-Fitting}.
We denote the set of all $m \times n$
matrices with entries in a ring $R$ by $M_{m \times n} (R)$ and in the case $m=n$
the group of all invertible elements of $M_{n \times n} (R)$ by $\Gl_n(R)$.
Let $A$ be a separable $K$-algebra and $\La$ be an $\fo$-order in $A$, finitely generated as $\fo$-module,
where $\fo$ is a complete commutative noetherian local ring with field of quotients $K$.
Moreover, we will assume that the integral closure of $\fo$ in $K$ is finitely generated as $\fo$-module.
The group ring $\zpg$ of a finite group $G$ will serve as a standard example.
Let $N$ and $M$ be two $\zeta(\La)$-submodules of
    an $\fo$-torsionfree $\zeta(\La)$-module.
    Then $N$ and $M$ are called {\it $\nr(\La)$-equivalent} if
    there exists an integer $n$ and a matrix $U \in \Gl_n(\La)$
    such that $N = \nr(U) \cdot M$, where $\nr: A \to \zeta(A)$ denotes
    the reduced norm map which extends to matrix rings over $A$ in the obvious way.
    We denote the corresponding equivalence class by $[N]_{\nr(\La)}$.
    We say that $N$ is
    $\nr(\La)$-contained in $M$ (and write $[N]_{\nr(\La)} \subset [M]_{\nr(\La)}$)
    if for all $N' \in [N]_{\nr(\La)}$ there exists $M' \in [M]_{\nr(\La)}$
    such that $N' \subset M'$. Note that it suffices to check this property for one $N_0 \in [N]_{\nr(\La)}$.
    We will say that $x$ is contained in $[N]_{\nr(\La)}$ (and write $x \in [N]_{\nr(\La)}$) if there is $N_0 \in [N]_{\nr(\La)}$ such that $x \in N_0$.\\

    Now let $M$ be a finitely presented (left) $\La$-module and let
    \beq \label{finite_representation}
        \La^a \stackrel{h}{\lto} \La^b \onto M
    \eeq
    be a finite presentation of $M$.
    We identify the homomorphism $h$ with the corresponding matrix in $M_{a \times b}(\La)$ and define
    $S(h) = S_b(h)$ to be the set of all $b \times b$ submatrices of $h$ if $a \geq b$. In the case $a=b$
    we call (\ref{finite_representation}) a quadratic presentation.
    The Fitting invariant of $h$ over $\La$ is defined to be
    $$\Fitt_{\La}(h) = \left\{ \barr{lll} [0]_{\nr(\La)} & \mbox{ if } & a<b \\
                        \left[\langle \nr(H) | H \in S(h)\rangle_{\zeta(\La)}\right]_{\nr(\La)} & \mbox{ if } & a \geq b. \earr \right.$$
    We call $\Fitt_{\La}(h)$ a Fitting invariant of $M$ over $\La$. One defines $\Fitt_{\La}^{\max}(M)$ to be the unique
    Fitting invariant of $M$ over $\La$ which is maximal among all Fitting invariants of $M$ with respect to the partial
    order ``$\subset$''. If $M$ admits a quadratic presentation $h$, one also puts $\Fitt_{\La}(M) := \Fitt_{\La}(h)$
    which is independent of the chosen quadratic presentation. \\

    Assume now that $\fo$ is an integrally closed commutative noetherian ring, but not necessarily complete or local.
    We denote by $\mathcal I = \mc I(\La)$ the $\zeta(\La)$-submodule of $\zeta(A)$ generated by the
    elements $\nr(H)$, $H \in M_{b\times b}(\La)$, $b \in \N$.
    We choose a maximal order $\La'$ containing $\La$. We may decompose the separable $K$-algebra $A$ into its simple components
    $$A = A_1 \op \cdots \op A_t,$$
    i.e.~each $A_i$ is a simple $K$-algebra and $A_i = A e_i = e_i A$ with central primitive idempotents $e_i$, $1 \leq i \leq t$.
    For any matrix $H \in M_{b \times b} (\La)$ there is a unique matrix
    $H^{\ast} \in M_{b\times b}(\La')$ such that   $H^{\ast} H = H H^{\ast} = \nr (H) \cdot 1_{b \times b}$
    and $H^{\ast} e_i = 0$ whenever $\nr(H) e_i =0$ (cf.~\cite{ich-Fitting}, Lemma 4.1;
    the additional assumption on $\fo$ to be complete local is not necessary). If
    $\ti H \in M_{b \times b} (\La)$ is a second matrix, then $(H \ti H)^{\ast} = \ti H^{\ast} H^{\ast}$.
    We define
    $$\mathcal H = \mathcal H(\La) := \left\{
        x \in \zeta(\La) |  x H^{\ast} \in M_{b \times b}(\La)  \forall b \in \N ~\forall H \in M_{b \times b} (\La)
         \right\}.$$
    Since $x \cdot \nr(H) = x H^{\ast} H$, we have in particular
    \beq \label{HI_in_zeta}
        \mathcal H \cdot \mathcal I = \mc H \subset \zeta(\La).
    \eeq
    We put $\mc H_p(G) := \mc H(\zpg)$ and $\mc H(G) := \mc H(\Z G)$.
    The importance of the $\zeta(\La)$-module $\mathcal H$ will become clear by means of the following result
    which is \cite{ich-Fitting}, Th.~4.2.
    \begin{theo} \label{annihilation-theo}
        If $\fo$ is an integrally closed complete commutative noetherian local ring
        and  $M$ is a finitely presented $\La$-module, then
        $$\mathcal H \cdot \Fitt_{\La}^{\max}(M) \subset \Ann_{\La}(M).$$
    \end{theo}

    Now let $C\punkt \in \mc D^{\perf}\tor(\La)$.
    If $\rho([C\punkt]) = 0$, we choose  $x \in K_1(A)$ such that $\partial(x) = [C\punkt]$ and define
    $$\Fitt_{\La}(C\punkt) := \left[\langle \nr_A (x) \rangle_{\zeta(\La)}\right]_{\nr(\La)}.$$
    It is straightforward to show that
    $$\Fitt_{\La}(C_2\punkt) = \Fitt_{\La}(C_1\punkt) \cdot \Fitt_{\La}(C_3\punkt)$$
    for any short exact sequence
    $C_1\punkt \into C_2\punkt \onto C_3\punkt$
    in $\mathcal C^b\tor(\PMod(\La))$, provided that all Fitting invariants are defined.
    Finally, if $C\punkt$ is isomorphic in $\mc D(\La)$ to a complex $P\me \to P^0$ concentrated in degree $-1$ and $0$ such that
    $P^i$ are finitely generated $\fo$-torsion $\La$-modules of finite projective dimension, $i = -1,0$, then
    $$\Fitt_{\La}(C\punkt) = \Fitt_{\La}(P^0 : P\me),$$
    where the righthand side denotes the relative Fitting invariant of \cite{ich-Fitting}, Def.~3.6.\\

    Now let $p \not=2$ be a prime and let $\La(\mc G)$ be the complete group algebra $\zp[[\mc G]]$,
    where $\mc G$ is a profinite group which contains a finite normal subgroup $H$ such that
    $\mc G/H \simeq \Ga$ for a pro-p-group $\Ga$, isomorphic to $\zp$; thus $\mc G$ can be written
    as a semi-direct product $H \rtimes \Ga$. We fix a topological
    generator $\ga$ of $\Ga$ and choose a natural number $n$ such that $\ga^{p^n}$ is central in $\mc G$.
    Since also $\Ga^{p^n} \simeq \zp$, there is an isomorphism $\zp [[\Ga^{p^n}]] \simeq \zp [[T]]$
    induced by $\ga^{p^n} \mapsto 1+T$. Here, $R := \zp [[T]]$ denotes the power series ring in one variable over $\zp$.
    If we view $\La(\mc G)$ as an $R$-module, there is a decomposition
    $$\La(\mc G) = \bigoplus_{i=0}^{p^n-1} R \ga^i [H].$$
    Hence $\La(\mc G)$ is finitely generated as an $R$-module and an $R$-order in the separable $Quot(R)$-algebra
    $\mathcal Q (\mc G) := \bigoplus_i Quot(R) \ga^i[H]$. Note that $\mc Q(\mc G)$ is obtained from $\La(\mc G)$ by inverting all non-zero
    elements in $R$.
    For any ring $\La$ and any $\La$-module $M$, we write $\pd_{\La}(M)$ for the projective dimension of $M$ over $\La$.
    For any finitely generated $\La(\mc G)$-module $M$, we write $\mu(M)$ for the
    Iwasawa $\mu$-invariant of $M$ considered as a $\La(\Ga)$-module.

    \begin{prop} \label{Fitt_becomes_int}
        Let $C\punkt$  be a complex in $\mc D^{\perf}\tor(\La(\mc G))$. Assume that $C\punkt$ is isomorphic in $\mc D(\La)$
        to a bounded complex $P\punkt$ such that $\pd_{\La(\mc G)}(P^j) \leq 1$, $\mu(P^j)=0$ and $P^j$ is $R$-torsion for all $j\in \N$.
        Assume that the Fitting invariant $\Fitt_{\qp \La(\mc G)}(\qp \otimes^{L} C\punkt)$ of $\qp \otimes^{L} C\punkt$ over $\qp \La(\mc G)$
        is generated by an element $\Phi \in \nr(K_1(\La_{(p)}(\mc G)))$, where the subscript $(p)$ means localization at the prime $(p)$.
        Then also
        $$\Fitt_{\La(\mc G)}(C\punkt) = [ \langle \Phi \rangle_{\zeta(\La(\mc G))} ]_{\nr(\La (\mc G))}.$$
    \end{prop}

    \begin{proof}
    We first observe that the homomorphism $\partial: K_1(\La(\mc G)) \to \kot{\La(\mc G)}$ is surjective (cf.~\cite{ich-Fitting}, Lemma 6.2
    or more directly \cite{Kakde-mc}, Lemma 5). Hence $\Fitt_{\La(\mc G)}(C\punkt)$ is defined for any complex in $\mc D^{\perf}\tor(\La(\mc G))$.
    Our assumptions on $C\punkt$  imply that we have an equality
    $$[C\punkt] = [P^{odd}] - [P^{ev}] \in \kot{\La(\mc G)}.$$
    Then $P^{odd}$ and $P^{ev}$ are two finitely generated $R$-torsion $\La(\mc G)$-modules of
    projective dimension less or equal to $1$ and trivial $\mu$-invariant.
    Let $\Psi$ be a generator of $\Fitt_{\La(\mc G)}(P^{ev})$. Since $P^{ev}$ vanishes after localization at $(p)$, we have
    $\Psi \in \nr(K_1(\La_{(p)}(\mc G)))$. But then $\Phi \cdot \Psi$ also belongs to $\nr(K_1(\La_{(p)}(\mc G)))$
    and is a generator of
    $$\Fitt_{\qp \La(\mc G)}(\qp \otimes^{L} C\punkt) \cdot \Fitt_{\qp \La(\mc G)}(\qp \otimes P^{ev}) = \Fitt_{\qp \La(\mc G)}(\qp \otimes P^{odd}).$$
    Now \cite{ich-tameII}, Prop.~3.2 implies that $\Phi \cdot \Psi$ is actually a generator of
    $\Fitt_{\La(\mc G)}(P^{odd})$ such that $\Phi$ is a generator of
    $$\Fitt_{\La(\mc G)}(C\punkt) = \Fitt_{\La(\mc G)}(P^{odd}) \cdot \Fitt_{\La(\mc G)}(P^{ev})\me.$$
    \end{proof}

    \subsubsection{Equivariant $L$-values}
    Let us fix a finite Galois extension $K/k$ of number fields with Galois group $G$.
    For any place $v$ of $k$ we fix a place $w$ of $K$ above $v$ and write $G_w$ resp.~$I_w$ for the decomposition
    group resp.~inertia subgroup of $K/k$ at $w$. Moreover, we denote the residual group at $w$ by $\ol G_w = G_w / I_w$
    and choose a lift $\phi_w \in G_w$ of the Frobenius automorphism at $w$. For a (finite) place $w$ we sometimes
    write $\fP_w$ for the associated prime ideal in $K$ and $\ord_w$ for the associated valuation.\\

    If $S$ is a finite set of places of $k$ containing the set $S_{\infty}$
    of all infinite places of $k$, and $\chi$ is a (complex) character of $G$, we denote the $S$-truncated Artin $L$-function
    attached to $\chi$ and $S$ by $L_S(s,\chi)$. Recall that there is a canonical isomorphism
    $\zeta(\C G) = \prod_{\chi \in \irr (G)} \C$, where $\irr (G)$ denotes the set
    of irreducible characters of $G$. We define the equivariant Artin $L$-function to be the
    meromorphic $\zeta(\C G)$-valued function
    $$L_S(s) := (L_S(s,\chi))_{\chi \in \irr (G)}.$$
    If $T$ is a second finite set of places of $k$ such that $S \cap T = \emptyset$, we define
    $\de_T(s) := (\de_T(s,\chi))_{\chi\in \irr (G)}$, where $\de_T(s,\chi) = \prod_{v \in T} \det(1 - N(v)^{1-s} \phi_w\me| V_{\chi}^{I_w})$
    and $V_{\chi}$ is a $G$-module with character $\chi$.
    We put
    $$\Theta_{S,T}(s) := \de_T(s) \cdot L_S(s)^{\sharp},$$
    where we denote by $^{\sharp}: \C G \to \C G$ the involution induced by $g \mapsto g\me$.
    These functions are the so-called $(S,T)$-modified $G$-equivariant $L$-functions and, for $r \in \Z_{\leq 0}$, we define Stickelberger elements
    $$\theta_S^T(K/k, r) = \theta_S^T(r) := \Theta_{S,T}(r) \in \zeta(\C G).$$
    If $T$ is empty, we abbreviate $\theta_S^T(r)$ by $\theta_S(r)$, and if $r=0$, we write $\theta_S^T$ for $\theta_S^T(0)$.
    Now a result of Siegel \cite{Siegel} implies that
    \beq \label{Stickelberger_is _rational}
        \theta_S^T(r) \in \zeta(\Q G)
    \eeq
    for all integers $r \leq 0$.
    Let us fix an embedding $\iota: \C \into \cp$; then
    the image of $\theta_S(r)$ in $\zeta(\qp G)$ via the canonical embedding
    $$\zeta(\Q G) \into \zeta (\qp G) = \bigoplus_{\chi \in \irrp(G) / \sim} \qp (\chi),$$
    is given by $\sum_{\chi} L_S(r, \check \chi^{\iota\me})^{\iota}$ and similarly for $\theta_S^T(r)$.
    Here, the sum runs over all $\cp$-valued irreducible characters of $G$ modulo Galois action.
    Note that we will frequently drop $\iota$ and $\iota\me$ from the notation. Finally, for an irreducible character $\chi$
    with values in $\C$ (resp.~$\cp$) we put
    $e_{\chi} = \frac{\chi(1)}{|G|} \sum_{g \in G} \chi(g\me) g$ which is a central idempotent in $\C G$ (resp.~$\cp G$).

    \subsubsection{Ray class groups}
    For any set $S$ of places of $k$, we write $S(K)$ for the set of places of $K$ which lie above those in $S$.
    Now let $T$ and $S$ be as above. We write $\cl_{K,T}$ for the ray class group of $K$ to the ray
    $\fM_T := \prod_{w \in T(K)} \fP_w$ and $\fo_S$ for the ring of $S(K)$-integers of $K$.
    Let $S_f$ be the set of all finite primes in $S(K)$;
    then there is a natural map $\Z S_f \to \cl_{K,T}$ which sends each prime $w\in S_f$
    to the corresponding class $[\fP_w] \in \cl_{K,T}$. We denote the cokernel of this map by $\cl_{S,T,K} =: \cl_{S,T}$.
    Further, we denote the $S(K)$-units of $K$ by $E_S$ and
    define
    $E_S^T := \left\{x \in E_S: x \equiv 1 \mod \fM_T \right\}$.
    All these modules are equipped with a
    natural $G$-action and we have the following exact sequences of $G$-modules
    \beq \label{ray_class_sequence_ZS}
        E_{S_{\infty}}^T \into E_S^T \stackrel{v}{\lto} \Z S_f \to \cl_{K,T} \onto \cl_{S,T},
    \eeq
    where $v(x) = \sum_{w\in S_f} \ord_w(x) w$ for $x \in E_S^T$,
    and
    \beq \label{ray_class_sequence}
        E_S^T \into E_S \to (\fo_S / \fM_T)\mal \stackrel{\nu}{\lto} \cl_{S,T} \onto \cl_S,
    \eeq
    where the map $\nu$ lifts an element $\ol x \in (\fo_S / \fM_T)\mal$ to $x \in \fo_S$ and
    sends it to the ideal class $[(x)] \in \cl_{S,T}$ of the principal ideal $(x)$.
    Note that the $G$-module $(\fo_S / \fM_T)\mal$ is c.t.~if
    no prime in $T$ ramifies in $K/k$.
    If $S = S_{\infty}$, we also write $E_K^T$ instead of $E_{S_{\infty}}^T$.
    Finally, we suppress the superscript $T$ from the notation if $T$ is empty. If $M$ is a finitely generated
    $\Z$-module and $p$ is a prime, we put $M(p) := \zp \otimes_{\Z} M$.
    In particular, we will be interested in $\cl_{K,T}(p)$ for odd primes $p$; we will abbreviate this module
    by $A_{K,T}$ if $p$ is clear from the context.

    \section{On different formulations of the equivariant Iwasawa main conjecture}

    The following reformulation of the EIMC was given in \cite{ich-tameII}, {\S}2.\\
    Let $p\not=2$ be a prime and let $\mc K/k$ be a Galois extension of totally real fields with
    Galois group $\mc G$, where $k$ is a number field, $\mc K$ contains the cyclotomic $\zp$-extension
    $k_{\infty}$ of $k$ and $[\mc K : k_{\infty}]$ is finite. Hence $\mc G$ is a $p$-adic Lie group of
    dimension $1$ and there is a finite normal subgroup $H$ of $\mc G$ such that
    $\mc G / H = \Gal(k_{\infty} / k) =: \Ga_k$. Here, $\Ga_k$ is isomorphic to the $p$-adic
    integers $\zp$ and we fix a topological generator $\ga_k$.
    If we pick a preimage $\ga$ of $\ga_k$ in $\mc G$, we can choose an integer $m$ such that
    $\ga^{p^m}$ lies in the center of $\mc G$. Hence the ring $R := \zp[[\Ga^{p^m}]]$ belongs
    to the center of $\La(\mc G)$, and $\La(\mc G)$ is an $R$-order in the separable
    $Quot(R)$-algebra $Q(\mc G)$.
    Let $S$ be a finite set of places of $k$
    containing all the infinite places $S_{\infty}$ and the set $S_p$ of all places of $k$ above $p$.
    Moreover, let $M_S$ be the maximal abelian pro-$p$-extension of $\mc K$ unramified outside $S$,
    and denote the Iwasawa module $\Gal(M_S / \mc K)$ by $X_S$. If $S$ additionally contains all places which ramify in $\mc K/k$,
    there is a canonical complex
    \beq \label{canonical_complex}
        C\punkt_S(\mc K/k): \dots  \to 0 \to C\me \to C^0\to 0 \to \dots
    \eeq
    of $R$-torsion $\La(\mc G)$-modules of projective dimension at most $1$ such that
    $H\me(C\punkt_S(\mc K/k)) = X_S$ and $H^0(C\punkt_S(\mc K/k)) = \zp$. For the moment we are insistent
    that this complex is the one constructed by Ritter and Weiss in \cite{towardI}. We will see later that we can
    work with $R\Hom(R\Ga_{\et}(\Spec(\fo_{\mc K}[\frac{1}{S}]), \qp / \zp), \qp / \zp)$ as well. We put (cf.~\cite{towardII}, {\S}4)
    $$\mho_S = \mho_S(\mc K/k) := (C\me) - (C^0) \in \kot{\La (\mc G)}.$$
    Since $\rho(\mho_S) = 0$, there is a well defined Fitting invariant of $\mho_S$; more precisely,
    $$\Fitt_{\La(\mc G)}(\mho_S) := \Fitt_{\La(\mc G)}(C\me : C^0) = \Fitt_{\La(\mc G)}(C\punkt_S(\mc K/k))\me.$$
    We recall some results concerning the algebra $Q (\mc G)$
    due to Ritter and Weiss \cite{towardII}.
    Let $\qpc$ be an algebraic closure of $\qp$ and fix an irreducible ($\qpc$-valued) character $\chi$ of $\mc G$
    with open kernel. Choose a finite field extension $E$ of $\qp$ such that the character $\chi$ has a realization
    $V_{\chi}$ over $E$.
    Let $\eta$ be an irreducible constituent of $\res^{\mc G}_H \chi$ and set
    $$St(\eta) := \{g \in \mc G: \eta^g = \eta \}, ~e_{\eta} = \frac{\eta(1)}{|H|} \sum_{g \in H} \eta(g\me) g,
    ~e_{\chi} = \sum_{\eta | \res^{\mc G}_H \chi} e_{\eta}.$$
    For any finite field extension $K$ of $\qp$ with ring of integers $\fo$,
    we set $Q^K(\mc G) := K \otimes_{\qp} Q(\mc G)$ and $\La^{\fo}(\mc G) = \fo[[\mc G]]$.
    By \cite{towardII}, corollary to Prop.~6, $e_{\chi}$ is a primitive central idempotent of $Q^E (\mc G)$.
    By loc.cit., Prop.~5 there is a distinguished element
    $\ga_{\chi} \in \zeta(Q^E (\mc G)e_{\chi})$ which generates a procyclic $p$-subgroup $\Ga_{\chi}$
    of $(Q^E (\mc G)e_{\chi})\mal$ and acts trivially on $V_{\chi}$.
    Moreover, $\ga_{\chi}$ induces an isomorphism $Q^E (\Ga_{\chi})
    \stackrel{\simeq}{\lto} \zeta( Q^E (\mc G)e_{\chi})$ by loc.cit., Prop.~6.
    For $r \in \N_0$, we define the following maps
    $$j_{\chi}^r: \zeta(Q^E (\mc G)) \onto \zeta(Q^E (\mc G)e_{\chi}) \simeq  Q^E(\Ga_{\chi}) \to  Q^E(\Ga_k),$$
    where the last arrow is induced by mapping $\ga_{\chi}$ to $\ka^r(\ga_{\chi}) \ga_k^{w_{\chi}}$,
    where $w_{\chi} = [\mc G : St(\eta)]$ and $\ka$ denotes the cyclotomic character of $\mc G$.
    Note that $j_{\chi} := j_{\chi}^0$ agrees with the corresponding map $j_{\chi}$ in loc.cit.
    It is shown that for any matrix $\Theta \in M_{n \times n} (Q(\mc G))$ we have
    \beq \label{jchi-und-det}
        j_{\chi} (\nr(\Theta)) = \det\,_{Q^E(\Ga_k)} (\Theta| \Hom_{EH}(V_{\chi},  Q^E(\mc G)^n)).
    \eeq
    Here, $\Theta$ acts on $f \in \Hom_{EH}(V_{\chi},  Q^E(\mc G)^n)$ via right multiplication,
    and $\ga_k$ acts on the left via $(\ga_k f)(v) = \ga_k \cdot f(\ga_k\me v)$ for all $v \in V_{\chi}$.
    Hence the map
    \bea
        \Det(~)(\chi): K_1(Q(\mc G)) & \to & Q^E(\Ga_k)\mal \\
        {[P,\al]}& \mapsto & \det\,_{Q^E(\Ga_k)} (\al| \Hom_{EH}(V_{\chi},  E \otimes_{\qp} P)),
    \eea
    where $P$ is a projective $Q(\mc G)$-module and $\al$ a $Q(\mc G)$-automorphism of $P$, is just $j_{\chi} \circ \nr$.
    If $\rho$ is a character of $\mc G$ of type $W$, i.e.~$\res^{\mc G}_H \rho = 1$, then we denote by
    $\rho^{\sharp}$ the automorphism of the field $Q\clo(\Ga_k) := \qpc \otimes_{\qp} Q(\Ga_k)$ induced by
    $\rho^{\sharp}(\ga_k) = \rho(\ga_k) \ga_k$. Moreover, we denote the additive group generated by all $\qpc$-valued
    characters of $\mc G$ with open kernel by $R_p(\mc G)$; finally, $\Hom^{\ast}(R_p( \mc G), Q\clo(\Ga_k)\mal)$
    is the group of all homomorphisms $f: R_p(\mc G) \to Q\clo(\Ga_k)\mal$ satisfying
    \[ \barr{ll}
        f(\chi \otimes \rho) = \rho^{\sharp}(f(\chi)) & \mbox{ for all characters } \rho \mbox{ of type } W \mbox{ and}\\
        f(\chi^{\sigma}) = f(\chi)^{\sigma} & \mbox{ for all Galois automorphisms } \sigma \in \Gal(\qpc/\qp).
    \earr\]
    We have an isomorphism
    \bea
        \zeta(Q(\mc G))\mal & \simeq & \Hom^{\ast}(R_p( \mc G), Q\clo(\Ga_k)\mal)\\
        x & \mapsto & [\chi \mapsto j_{\chi}(x)].
    \eea
    By loc.cit., Th.~5 the map $\Theta \mapsto [\chi \mapsto \Det(\Theta)(\chi)]$ defines a homomorphism
    $$\Det: K_1(Q(\mc G)) \to \Hom^{\ast}(R_p(\mc G), Q\clo(\Ga_k)\mal)$$
    such that we obtain a commutative triangle
    \beq \label{Det_triangle}\xymatrix{
        & K_1(Q(\mc G)) \ar[dl]_{\nr} \ar[dr]^{\Det} &\\
        {\zeta(Q(\mc G))\mal} \ar[rr]^{\sim} & & {\Hom^{\ast}(R_p( \mc G), Q\clo(\Ga_k)\mal)}.
    }\eeq
    We put $u := \ka(\ga_k)$ and fix a finite set $S$ of places of $k$ containing $S_{\infty}$ and all places which ramify in $\mc K/k$.
    Each topological generator $\ga_k$ of  $\Ga_k$ permits the definition of a power series
    $G_{\chi,S}(T) \in \qpc \otimes_{\qp} Quot(\zp[[T]])$ by starting out from the Deligne-Ribet
    power series for abelian characters of open subgroups of $\mc G$ (cf.~\cite{Deligne-Ribet}).
    One then has an equality
    $$L_{p,S}(1-s,\chi) = \frac{G_{\chi,S}(u^s-1)}{H_{\chi}(u^s-1)},$$
    where $L_{p,S}(s,\chi)$ denotes the $p$-adic Artin $L$-function, and where, for irreducible $\chi$,
    one has
    $$H_{\chi}(T) = \left\{\barr{ll} \chi(\ga_k)(1+T)-1 & \mbox{ if }  H \subset \ker (\chi)\\
                    1 & \mbox{ otherwise.}  \earr\right.$$
    Now \cite{towardII}, Prop.~11 implies that
    $$L_{k,S} : \chi \mapsto \frac{G_{\chi,S}(\ga_k-1)}{H_{\chi}(\ga_k-1)}$$
    is independent of the topological generator $\ga_k$ and lies in $\Hom^{\ast}(R_p( \mc G), Q\clo(\Ga_k)\mal)$.
    Diagram (\ref{Det_triangle}) implies that there is a unique element $\Phi_S \in \zeta(Q(\mc G))\mal$
    such that
    $$j_{\chi}(\Phi_S) = L_{k,S}(\chi).$$
    The EIMC as formulated in \cite{towardII} now states that
    there is a  unique $\Theta_S \in K_1(Q(\mc G))$ such that $\Det(\Theta_S) = L_{k,S}$ and
    $\partial(\Theta_S) = \mho_S$.
    The EIMC without its uniqueness statement hence asserts that there is $x \in K_1(Q(\mc G))$
    such that $\partial(x) = \mho_S$ and $\Det(x) = L_{k,S}$;
    now diagram (\ref{Det_triangle}) implies that $\nr(x) = \Phi_S$, and thus $\Phi_S$
    is a generator of $\Fitt_{\La(\mc G)}(\mho_S)$. Conversely, if $\Phi_S$ is a generator of $\Fitt_{\La(\mc G)}(\mho_S)$,
    then there is an element $x \in K_1(Q(\mc G))$ such that $\partial(x) = \mho_S$ and $\langle \nr(x) \rangle_{\zeta(\La(\mc G))}$ is
    $\nr(\La(\mc G))$-equivalent to $\langle \Phi_S \rangle_{\zeta(\La(\mc G))}$, i.e.~there is an $u \in K_1(\La(\mc G))$ such that
    $\nr (x) = \nr (u) \cdot \Phi_S$. But then $\Theta_S := x \cdot u\me$ has $\partial(\Theta_S) = \partial (x) = \mho_S$
    and $\Det (\Theta_S) = L_{k,S}$, since $\nr(\Theta_S) = \Phi_S$. We have shown that the following conjecture is equivalent
    to the EIMC without the uniqueness of $\Theta_S$.

    \begin{con} \label{EIMC}
        The element $\Phi_S \in \zeta(Q(\mc G))\mal$ is a generator of $\Fitt_{\La(\mc G)}(\mho_S)$.
    \end{con}

    The following theorem is due to Ritter and Weiss \cite{EIMC-theorem}:

    \begin{theo} \label{EIMC-Ritter-Weiss}
        Conjecture \ref{EIMC} is true provided that the $\mu$-invariant $\mu(X_S)$ vanishes.
    \end{theo}

    We also discuss Conjecture \ref{EIMC} within the framework of the theory of \cite{CFKSV}, {\S}3.
    For this, let
    $$\pi: \mc G \to \Gl_n(\fo_E)$$
    be a continuous homomorphism, where $\fo_E$ denotes the ring of integers of $E$
    and $n$ is some integer greater or equal to $1$.
    There is a ring homomorphism
    \beq \label{first_Phi}
        \Phi_{\pi}: \La(\mc G) \to M_{n\times n}(\La^{\fo_E}(\Ga_k))
    \eeq
    induced by the continuous group homomorphism
    \bea
        \mc G & \to & (M_{n \times n}(\fo_E) \otimes_{\zp} \La(\Ga_k))\mal = \Gl_n(\La^{\fo_E}(\Ga_k))\\
        \si & \mapsto & \pi(\si) \otimes \ol \si,
    \eea
    where $\ol \si$ denotes the image of $\si$ in $\mc G / H = \Ga_k$. By loc.cit., Lemma 3.3 the
    homomorphism (\ref{first_Phi}) extends to a ring homomorphism
    $$\Phi_{\pi}: Q(\mc G) \to M_{n\times n}(Q^E(\Ga_k))$$
    and this in turn induces a homomorphism
    $$\Phi_{\pi}': K_1(Q(\mc G)) \to K_1(M_{n\times n}(Q^E(\Ga_k))) = Q^E(\Ga_k)\mal.$$
    Let $\aug: \La^{\fo_E}(\Ga_k) \onto \fo_E$ be the augmentation map and put $\fp = \ker(\aug)$.
    Writing $\La^{\fo_E}(\Ga_k)_{\fp}$ for the localization of $\La^{\fo_E}(\Ga_k)$ at $\fp$, it is clear that
    $\aug$ naturally extends to a homomorphism $\aug: \La^{\fo_E}(\Ga_k)_{\fp} \to E$.
    One defines an evaluation map
    \bea
    \phi: Q^E(\Ga_k) & \to & E \cup \{\infty\}\\
    x & \mapsto & \left\{ \barr{ll} \aug (x) & \mbox{ if } x \in \La^{\fo_E}(\Ga_k)_{\fp}\\
                        \infty & \mbox{ otherwise}. \earr \right.
    \eea
    If $\Theta$ is an element of $K_1(Q(\mc G))$, we define $\Theta(\pi)$ to be $\phi(\Phi_{\pi}'(\Theta))$.
    We need the following lemma.
    \begin{lem}
        If $\pi = \pi_{\chi}$ is a representation of $\mc G$ with character $\chi$ and $r\in \N_0$, then
        \[ \xymatrix{
            K_1(Q(\mc G)) \ar[r]^-{\Phi_{\pi_{\chi}\ka^r}'} \ar[d]^{\nr} & K_1(M_{n \times n}(Q^E(\Ga_k))) \ar[d]^{\nr}_{\simeq}\\
            {\zeta(Q(\mc G))\mal} \ar[r]^{j_{\chi}^r} & Q^E(\Ga_k)\mal
        } \]
        commutes. In particular, we have $\nr \circ \Phi_{\pi_{\chi}}' = \Det(~)(\chi)$.
    \end{lem}
    \begin{proof}
    This is \cite{ich-tameII}, lemma 2.3.
    \end{proof}

    Conjecture \ref{EIMC} now implies that there is an element $\Theta_S \in K_1(Q(\mc G))$ such that
    $\partial(\Theta_S) = \mho_S$ and for any $r \geq 1$ divisible by $p-1$ we have
    $$\Theta_S(\pi_{\chi}\ka^r) = \phi(j_{\chi}^r(\Phi_S)) = L_S(1-r,\chi).$$

    The following result explains, why we may replace the complex (\ref{canonical_complex}) by the complex
    $R\Hom(R\Ga_{\et}(\Spec(\fo_{\mc K}[\frac{1}{S}]), \qp / \zp), \qp / \zp)$. Though it might  be no surprise
    to experts, the author is not aware of any reference for this result.

    \begin{theo}
    With the notation as above, there is an isomorphism
    $$C\punkt_S(\mc K/k) \simeq R\Hom(R\Ga_{\et}(\Spec(\fo_{\mc K}[\frac{1}{S}]), \qp / \zp), \qp / \zp)$$
    in $\mc D(\La(\mc G))$. In particular, there is an equality
    $$\mho_S = - [R\Hom(R\Ga_{\et}(\Spec(\fo_{\mc K}[\frac{1}{S}]), \qp / \zp), \qp / \zp)] \in \kot{\La(\mc G)}.$$
    \end{theo}

    \begin{proof}
    Since $\qp / \zp$ is a direct limit of finite abelian groups of $p$-power order, we have an isomorphism with Galois cohomology
    \beq \label{et-to-gal-coh}
        R\Ga_{\et}(\Spec(\fo_{\mc K}[\frac{1}{S}]), \qp / \zp) \simeq R\Ga(X_S, \qp / \zp).
    \eeq
    We put $G_S := \Gal(M_S / k)$. Now for any compact (right) $\La(G_S)$-modules $M$ and discrete (left) $\La(G_S)$-module $N$
    (considered as complexes in degree zero), there is an isomorphism
    $$M \otimes^L_{\La(X_S)} R\Hom(N, \qp / \zp) \simeq R\Hom(R\Hom_{\La(X_S)}(M,N), \qp / \zp)$$
    in $\mc D(\La(\mc G))$ (cf.~\cite{NSW}, Cor.~5.2.9 or \cite{Weibel-int-hom}, Th.~10.8.7). Noting that
    $R\Hom_{\La(X_S)}(\zp, N)$ identifies with $R\Ga(X_S, N)$ we specialize $M = \zp$ and $N = \qp / \zp$ which yields an isomorphism
    \beq \label{zpzp-to-gal-coh}
        \zp \otimes^L_{\La(X_S)} \zp \simeq R\Hom(R\Ga(X_S, \qp / \zp), \qp / \zp)
    \eeq
    in $\mc D(\La(\mc G))$. Now we consider the short exact sequence
    $\De G_S \into \La(G_S) \onto \zp$
    of left $\La(G_S)$-modules, where the surjection is the augmentation map and $\De(G_S)$ denotes its kernel. We now apply
    $\zp \hat{\otimes}_{\La(X_S)} \_$ to this sequence and obtain the following exact homology sequence:
    \[ \xymatrix{
        {H_1(X_S,\zp)} \ar[d]^{\simeq} \inj[r] & H_0(X_S, \De (G_S)) \ar[r] \equal[d] & H_0(X_S, \La(G_S)) \sur[r] \equal[d] & H_0(X_S, \zp) \equal[d]\\
        X_S \inj[r] & \zp \hat{\otimes}_{\La(X_S)} \De (G_S) \ar[r] &  \La(\mc G) \sur[r] & \zp
    }\]
    In particular, we find that
    $$H_i(X_S, \De (G_S)) = H_i(X_S, \La(G_S)) = 0 \mbox{ for all } i > 0.$$
    Hence the exact triangle
    $$\zp \otimes^L_{\La(X_S)} \De (G_S) \to \zp \otimes^L_{\La(X_S)} \La (G_S) \to \zp \otimes^L_{\La(X_S)} \zp \to$$
    implies that the complex
    \beq \label{realizationcomplex}
        \zp \hat{\otimes}_{\La(X_S)} \De (G_S) \to  \La(\mc G)
    \eeq
    of the above homology sequence is isomorphic to $\zp \otimes^L_{\La(X_S)} \zp$
    in $\mc D(\La(\mc G))$, and hence also to $R\Hom(R\Ga_{\et}(\Spec(\fo_{\mc K}[\frac{1}{S}]), \qp / \zp), \qp / \zp)$ using (\ref{et-to-gal-coh})
    and (\ref{zpzp-to-gal-coh}). Let $\De(G_S, X_S)$ be the closure of the right $\La(G_S)$-ideal generated by $x-1$, $x \in X_S$. Then
    $$\zp \hat{\otimes}_{\La(X_S)} \De (G_S) = \De(G_S) / \De(G_S,X_S) \De(G_S) =: Y_S$$
    and the map in (\ref{realizationcomplex}) is induced by mapping $g-1$ to $\ol g -1$, where $\ol g$ denotes the image
    of an element $g \in G_S$ in $\mc G$ under the canonical projection $G_S \onto \mc G$. But the translation functor of
    Ritter and Weiss transfers the exact sequence $X_S \into G_S \onto \mc G$ into
    $$X_S \into Y_S \onto \De (\mc G),$$
    where the projection is induced in exactly the same way. Hence if we glue this sequence with the natural augmentation sequence
    $\De(\mc G) \into \La(\mc G) \onto \zp$,
    we obtain exactly the homology sequence above. The result now follows, once we observe that the complex (\ref{canonical_complex}) of Ritter and
    Weiss is achieved by a commutative diagram
    \[ \xymatrix{
        & \La(\mc G) \inj[d] \equal[r] & \La(\mc G) \inj[d] & \\
        X_S \inj[r] \equal[d] & Y_S \sur[d] \ar[r] & \La(\mc G) \sur[r] \sur[d] & \zp \equal[d] \\
        X_S \inj[r] & C\me \ar[r] & C^0 \sur[r] & \zp
    }\]
    \end{proof}

    \begin{rem}
        For more information concerning the work of Ritter and Weiss in comparison with Kakde's approach, the reader may consult
        Venjakob's recent survey article \cite{Venjakob-comparison}.
    \end{rem}

    \section{Another reformulation of the equivariant Iwasawa main conjecture}

    We first recall the notion of abstract $1$-motives and their basic properties as formulated in \cite{EIMC-reform}.
    For an arbitrary abelian group $J$ and a positive integer $m$, we denote by $J[m]$ the maximal $m$-torsion subgroup of $J$.
    For a prime $p$ the $p$-adic Tate module of $J$ is defined to be
    $$T_p(J) := \varprojlim_n J[p^n],$$
    where the projective limit is taken with respect to multiplication by $p$. An abelian, divisible group $J$ is of finite
    local corank if for any prime $p$ there is an integer $r_p(J)$ and a $\zp$-isomorphism
    $$J[p^{\infty}] \simeq (\qp / \zp)^{r_p(J)},$$
    where $J[p^{\infty}] = \bigcup_n J[p^n]$.

    \begin{defi}
    An abstract $1$-motive $\mc M := [L \stackrel{\de}{\lto} J]$ consists of the following data.
    \bit
        \item
        a free $\Z$-module $L$ of finite rank;
        \item
        an abelian, divisible group $J$ of finite local corank;
        \item
        a group morphism $\de: L \to J$.
    \eit
    \end{defi}

    Now let $n \geq 1$ be an integer and consider the fiber product $J \times^n_J L$ with respect to the map
    $\de$ and the multiplication by $n$ map $J \stackrel{n} {\onto} J$.

    \begin{defi}
    The group $\mc M[n] := (J \times^n_J L) \otimes \Z / n \Z$ is called the group of $n$-torsion points of $\mc M$.
    Moreover, if $n \mid m$, there are canonical surjective multiplication by $m/n$ maps $\mc M[m] \to \mc M[n]$ and we define the $p$-adic
    Tate module of $\mc M$ to be
    $$T_p(\mc M) := \varprojlim_n \mc M[p^n].$$
    \end{defi}

    In this way, we obtain for every prime $p$ an exact sequence of free $\zp$-modules
    \beq \label{ses-1-motive}
        T_p(J) \into T_p(\mc M) \onto \zp \otimes L.
    \eeq
    We are now going to define an Iwasawa theoretic abstract $1$-motive. Again, we follow the treatment in \cite{EIMC-reform}.
    For this, let $p$ be an odd prime and $\mc K$ be the cyclotomic $\zp$-extension of a number field $K$ and let $K_n$ denote its $n$-th layer, $n \in \N$.
    We denote the set of $p$-adic
    places of $\mc K$ by $\mc S_p$ and fix two finite sets $\mc S$ and $\mc T$ of places of $\mc K$ such that $\mc T \cap (\mc S \cup \mc S_p) = \emptyset$.
    The divisor group of $\mc K$ is given by
    $$\Div_{\mc K} := \bigoplus_v  \Ga_v \cdot v,$$
    where the direct sum runs over all finite primes of $\mc K$ and $\Ga_v = \Z$ (resp.~$\Ga_v = \Z[\frac{1}{p}]$) if $v \not\in \mc S_p$
    (resp.~$v \in \mc S_p$). Note that in both cases $\Ga_v$ identifies with the value group of an appropriate chosen valuation $\ord_v$ corresponding to $v$.
    We let
    $$\Div_{\mc K,\mc T} := \bigoplus_{v \not\in \mc T}  \Ga_v \cdot v,~\mc K_{\mc T}\mal := \left\{x \in \mc K\mal \mid \ord_v(x-1)>0 ~\forall v\in \mc T\right\}.$$
    The usual divisor map induces a group morphism
    $$\mathrm{div}_{\mc K}: \mc K_{\mc T}\mal \to \Div_{\mc K,\mc T}, ~x \mapsto \sum_v \ord_v(x)$$
    and we define a generalized ideal class group by
    $$C_{\mc K,\mc T} := \frac{\Div_{\mc K, \mc T}}{\mathrm{div}_{\mc K}(\mc K_{\mc T}\mal)}.$$
    Then the classical Iwasawa $\mu$-invariant attached to $\mc K$ and $p$ vanishes if and only if $\mc A_{\mc K, \mc T} := C_{\mc K,\mc T} \otimes \zp$ is divisible of finite local corank.
    Note that the vanishing of $\mu$ only depends upon $\mc K$ and $p$, but not on the number field $K$. We will henceforth assume that $\mu=0$
    and associate to the data $(\mc K, \mc S, \mc T)$ the abstract $1$-motive
    $$\mc M_{\mc S, \mc T}^{\mc K} := [\Div_{\mc K}(\mc S \sm \mc S_p) \stackrel{\de}{\lto} \mc A_{\mc K, \mc T}],$$
    where $\Div_{\mc K}(\mc S \sm \mc S_p)$ is the group of divisors of $\mc K$ supported on $\mc S \sm \mc S_p$ and $\de$ is induced
    by the usual divisor class map. In particular, the exact sequence (\ref{ses-1-motive}) now reads
    \beq \label{ses-Iwasawa-motive}
        T_p(\mc A_{\mc K, \mc T}) \into T_p(\mc M_{\mc S, \mc T}^{\mc K}) \onto \Div_{\mc K}(\mc S \sm \mc S_p) \otimes \zp.
    \eeq

     Now assume that $\mc K$ is a Galois extension of a totally real number field $k$,
     and that $\mc K$ is the cyclotomic $\zp$-extension of a CM-number field $K$.
     Let $\mc G = \Gal(\mc K/k)$ and let $j \in \mc G$ denote the unique central automorphism in $\mc G$ which is induced by complex conjugation.
     Let $\mc K^+$ be the maximal real subfield of $\mc K$ and $\mc G^+ = \mc G / \langle j \rangle$ its Galois group over $k$.
     We fix two finite, non-empty, disjoints sets $S$ and $T$ of places in $k$, such that $S$
     contains $S\ram(\mc K/k) \cup S_{\infty}$. Let $\mc S$ and $\mc T$ be the sets of finite primes in $\mc K$ sitting above primes in $S$ and
     $T$, respectively.
     Then by \cite{EIMC-reform}, Th. 4.6 the Tate module $T_p(\mc M_{\mc S, \mc T}^{\mc K})^-$ is a $\zp$-free torsion $\La(\mc G)_- := \La(\mc G) / (1+j)$-module of projective dimension at most $1$.
     For $v \in T$ we put
     $$\xi_v := \nr(\ka(\phi_{w}) - \phi_{w}),$$
     where $\phi_{w} \in \mc G$ denotes the Frobenius at a chosen prime $w$ in $\mc K$ above $v$. Let $x \mapsto \dot{x}$ be the
     automorphism on $\La(\mc G)$ induced by $g \mapsto \ka(g)g\me$ for $g \in \mc G$. We put
     $$\Psi_{S,T} = \Psi_{S,T}(\mc K / k) := \prod_{v \in T} \xi_v \cdot \dot \Phi_S.$$
     For any positive integer $n$, we denote by $\zeta_n$ a primitive $n$-th root of unity.
     We are now ready to prove the following variant of the equivariant Iwasawa main conjecture which generalizes \cite{EIMC-reform}, Th.~5.6
     to the non-abelian situation.

     \begin{theo}\label{EIMC-GP}
        Let $(\mc K / k, S, T, p)$ be as above. If Iwasawa's $\mu$-invariant attached to $\mc K(\zeta_p)$ and $p$ vanishes, then
        $\Psi_{S,T}$ is a generator of $\Fitt_{\La(\mc G)_-}(T_p(\mc M_{\mc S, \mc T}^{\mc K})^-)$.
     \end{theo}

     \begin{proof}
     We first remark that $T_p(\mc M_{\mc S, \mc T}^{\mc K})^-$ admits a quadratic presentation by
     \cite{ich-Fitting}, Lemma 6.2. Hence $\Fitt_{\La(\mc G)_-}(T_p(\mc M_{\mc S, \mc T}^{\mc K})^-)$ is well defined.\\
     If $M$ is an Iwaswa torsion module, we write $\al(M)$ for the Iwasawa adjoint of $M$. We will need the following proposition.

     \begin{prop} \label{Fitt-alpha}
        Assume that $C$ is a finitely generated $R$-torsion $\La(\mc G^+)$-module of projective dimension
        at most $1$ which has no nontrivial finite submodule and that $\Phi$ is a generator of $\Fitt_{\La(\mc G^+)}(C)$;
        then $\Fitt_{\La(\mc G)}(\al(C)(1))$ is generated by $\dot \Phi e^- + e^+$, where $e^{\pm} = \frac{1 \pm j}{2}$.
     \end{prop}
     \begin{proof}
        Let $\psi: \La(\mc G^+)^m \to \La(\mc G^+)^m$ be a quadratic presentation of $C$ such that $\nr(\psi) = \Phi$.
        By \cite{ich-Fitting}, Prop.~6.3 (i) resp.~its proof it follows that $\psi^{T,\sharp}$ is a finite presentation of $\al(C)$
        and $\nr(\psi^{T,\sharp}) = \Phi^{\sharp}$ is a generator of $\Fitt_{\La(\mc G^+)}(\al(C))$, where
        $\psi^T$ denotes the transpose of $\psi$.
        Now $\La(\mc G^+) \simeq \La(\mc G) e_+$ and the involution $g \mapsto \ka(g\me) g$ induces
        an isomorphism between the first Tate twist of $\La(\mc G^+)$ and $\La(\mc G)e_-$. We obtain a quadratic presentation
        $\dot \psi^T: (\La(\mc G)e_-)^m \to (\La(\mc G)e_-)^m$ of $\al(C)(1)$ regarded as $\La(\mc G)e_-$-module. Since
        $\nr(\dot \psi^T) = \dot \Phi$ and $\al(C)(1)$ is trivial on plus parts, we are done.
     \end{proof}

     Returning to the proof of Theorem \ref{EIMC-GP}, we let
     $$\De_{\mc K, \mc T} := \bigoplus_{w \in \mc T} \ka(w)\mal,$$
     where $\ka(w)$ denotes the residue field at $w$. We first assume that $\zeta_p$ lies in $\mc K$.
     Then there is an exact sequence of $\La(\mc G)$-modules (cf.~\cite{EIMC-reform}, remark 3.10)
     \beq \label{GP-sequence}
        \zp(1) \into T_p(\De_{\mc K, \mc T})^- \to T_p(\mc M_{\mc S, \mc T}^{\mc K})^- \onto T_p(\mc M_{\mc S, \emptyset}^{\mc K})^-.
     \eeq
     Applying $\al(\_)(1)$ to this sequence yields a complex
     $$\ti C\punkt(\mc K^+/k): \al(T_p(\mc M_{\mc S, \mc T}^{\mc K})^-)(1) \to \al(T_p(\De_{\mc K, \mc T})^-)(1)$$
     whose non-trivial cohomology groups are given by $H\me(\ti C\punkt(\mc K^+/k)) = X_S^+$ and $H^0(\ti C\punkt(\mc K^+/k) = \zp$;
     here, we have used \cite{EIMC-reform}, Lemma 3.9. Most likely, this complex is isomorphic in
     $\mc D(\La(\mc G^+))$ to the canonical complex $C\punkt(\mc K^+/k)$, but we will not need this.
     Since we assume that $\mu = 0$, we have
     $$\Fitt_{\mc R}(\mc R \otimes^L_{\La (\mc G^+)} \ti C\punkt(\mc K^+/k)) = \Fitt_{\mc R}(\mc R \otimes^L_{\La (\mc G^+)} C\punkt(\mc K^+/k)),$$
     where $\mc R$ is either $\qp \otimes \La(\mc G^+)$ or $\La_{(p)}(\mc G^+)$. But then Proposition \ref{Fitt_becomes_int}
     implies that the two Fitting invariants also agree for $\mc R = \La(\mc G^+)$. Hence
     \bea
        \Fitt_{\La(\mc G^+)}(\al(T_p(\mc M_{\mc S, \mc T}^{\mc K})^-)(1))
        & = & \Fitt_{\La(\mc G^+)}(\al(T_p(\De_{\mc K, \mc T})^-)(1)) \cdot \Fitt_{\La(\mc G^+)}(\ti C\punkt(\mc K^+/k))\me\\
        & = & \Fitt_{\La(\mc G^+)}(\al(T_p(\De_{\mc K, \mc T})^-)(1)) \cdot \Fitt_{\La(\mc G^+)}(C\punkt(\mc K^+/k))\me\\
        & = & \prod_{v \in T} \dot\xi_v \cdot \Phi_S \\
        & = & \dot\Psi_{S,T},
     \eea
     where we have used Theorem \ref{EIMC-Ritter-Weiss} and Proposition \ref{Fitt-alpha}. Note that our assumption on the vanishing of $\mu$
     is in fact equivalent to $\mu(X_S^+) = 0$. The result follows by applying Proposition \ref{Fitt-alpha} again.
     If $\zeta_p \not\in \mc K$, let $\ti{\mc K} := \mc K(\zeta_p)$ and let $\ti{\mc S}$ (resp.~$\ti{\mc T}$) be the set of places of $\ti{\mc K}$
     above those in $\mc S$ (resp.~$\mc T$). Let $\ti{\mc G}$ denote the Galois group $\Gal(\ti{\mc K} / k)$. By \cite{EIMC-reform}, Cor.~4.8
     we have an isomorphism of $\La(\mc G)_-$-modules
     $$T_p(\mc M_{\mc S, \mc T}^{\mc K})^- \simeq \La(\mc G)_- \otimes_{\La(\ti{\mc G})_-} T_p(M_{\ti{\mc S}, \ti{\mc T}}^{\ti{\mc K}})^-.$$
     If $\pi$ denotes the canonical epimorphism $\La(\ti{\mc G})_- \onto \La(\mc G)_-$, the natural behavior of Fitting invariants thus gives an equality
     $$\Fitt_{\La(\mc G)_-} (T_p(\mc M_{\mc S, \mc T}^{\mc K})^-) = \pi \left(\Fitt_{\La(\ti{\mc G})_-} (T_p(M_{\ti{\mc S}, \ti{\mc T}}^{\ti{\mc K}})^-)\right).$$
     This suffices, since the natural behavior of $p$-adic $L$-functions gives
     $$\pi (\Psi_{S,T}(\ti{\mc K}/k) = \Psi_{S,T}(\mc K/k).$$
     \end{proof}

     \begin{rem}
     Note that the argument can be reversed to show that $\Psi_{S,T}$ is a generator of $\Fitt_{\La(\mc G)_-}(T_p(\mc M_{\mc S, \mc T}^{\mc K})^-)$
     if and only if the EIMC (Conjecture \ref{EIMC}) holds (provided that $\mu=0$).
     \end{rem}

     \section{The non-abelian Brumer-Stark conjecture}

     Let $K/k$ be a finite Galois CM-extension with Galois group $G$.
    Let $S$ and $T$ be two finite sets of places of $k$ such that
    \bit
        \item
        $S$ contains all the infinite places of $k$ and all the places which ramify in $K/k$,
        i.e.~$S \supset S_{\ram} \cup S_{\infty}$.
        \item
        $S \cap T = \emptyset$.
        \item
        $E_S^T$ is torsionfree.
    \eit
    We refer to the above hypotheses as $Hyp(S,T)$.
    We put $\La = \Z G$ and choose a maximal order $\La'$ containing $\La$.
    For a fixed set $S$ we define $\fA_S$ to be the $\zeta(\La)$-submodule of $\zeta(\La')$ generated
    by the  elements $\de_T(0)$, where $T$ runs through the finite sets of places of $k$
    such that $Hyp(S,T)$ is satisfied.
    The following conjecture has been formulated in \cite{ich-stark} and is a non-abelian generalization of Brumer's conjecture.

    \begin{con}[$B(K/k,S)$] \label{Brumer}
        Let $S$ be a finite set of places of $k$ containing $S_{\ram} \cup S_{\infty}$. Then
        $\fA_S \theta_S \subset \mathcal I(G)$ and for each $x \in \mathcal H(G)$
        we have
        $$x \cdot \fA_S \theta_S \subset \Ann_{\La} (\cl_K).$$
    \end{con}

    \begin{rem}
        \bit
        \item
            If $G$ is abelian, \cite{Tate-Stark}, Lemma 1.1 p.~82 implies that
            the module $\fA_S$ equals $\Ann_{\Z G}(\mu_K)$. In this case the inclusion
            $\fA_S \theta_S \subset \mathcal I(G) = \La = \Z G$ holds by (\ref{Ca-DR-Ba}) and,
            since $\mathcal H(G) = \La$ in this case, Conjecture \ref{Brumer} recovers
            Brumer's conjecture.
        \item
            Replacing the class group $\cl_K$ by its $p$-parts $\cl_K(p)$ for each rational prime $p$,
            Conjecture $B(K/k,S)$ naturally decomposes into local conjectures $B(K/k,S,p)$.
            Note that it is possible to replace $\mc H(G)$ by $\mc H_p(G)$ by \cite{ich-stark}, Lemma 1.4.
        \item
            Burns \cite{Burns-derivatives} has also formulated a conjecture which generalizes many refined Stark conjectures to the
            non-abelian situation. In particular, it implies our generalization of Brumer's conjecture
            (cf.~loc.cit., Prop.~3.5.1).
        \eit
    \end{rem}

    For $\al \in K\mal$ we define
    $$S_{\al} :=  \{v \mbox{ finite place of } k: \fp_v | N_{K/k}(\al) \}$$
    and we call $\al$ an {\it anti-unit} if $\al^{1+j} = 1$.
    Let $\om_K := \nr (|\mu_K|)$. The following is a non-abelian generalization of the
    Brumer-Stark conjecture (cf.~\cite{ich-stark}, Conj.~2.6).

    \begin{con}[$BS(K/k,S)$] \label{Brumer-Stark}
        Let $S$ be a finite set of places of $k$ containing $S_{\ram} \cup S_{\infty}$.
        Then $\om_K \cdot \theta_S \in \mathcal I(G)$ and for each $x \in \mathcal H(G)$
        and each fractional ideal $\fa$ of $K$, there is an anti-unit $\al = \al(x,\fa,S) \in K\mal$ such that
        $$\fa^{x \cdot \om_K \cdot \theta_S} = (\al)$$
        and for each finite set $T$ of primes of $k$ such that $Hyp(S \cup S_{\al},T)$ is satisfied
        there is an $\al_T \in E_{S_{\al}}^T$
        such that
        \beq \label{abelian-ersatz}
            \al^{z \cdot \de_T(0)} = \al_T^{z \cdot \om_K}
        \eeq
         for each $z \in \mc H(G)$.
    \end{con}

    \begin{rem}
    \bit
        \item
        If $G$ is abelian, we have $\mc I(G) = \mc H(G) = \Z G$ and $\om_K = |\mu_K|$.
        Hence it suffices to treat the case $x=z=1$.
        Then \cite{Tate-Stark}, Prop.~1.2, p.~83 states that the condition (\ref{abelian-ersatz})
        on the anti-unit $\al$ is equivalent to the assertion that the extension $K(\al^{1/\om_K}) / k$
        is abelian.
        \item
        As above, we obtain local conjectures $BS(K/k,S,p)$ for each prime $p$.
    \eit
    \end{rem}

    For any $\zpg$-module $M$, we denote the Pontryagin dual $\Hom (M, \qp / \zp)$ of $M$ by $M^{\vee}$ which
    is endowed with the contravariant $G$-action $(gf)(m) = f (g\me m)$ for $f \in M^{\vee}$, $g \in G$ and $m \in M$.
    As a consequence of Theorem \ref{EIMC-GP} we now prove the following non-abelian generalization of \cite{EIMC-reform}, Th.~6.5.

    \begin{theo} \label{strongBS}
    Let $K/k$ be a Galois CM-extension with Galois group $G$ and $p$ an odd prime. Fix two finite sets $S$ and $T$ of primes of $k$
    such that $Hyp(S,T)$ is satisfied. If Iwasawa's $\mu$-invariant attached to the cyclotomic $\zp$-extension of $K(\zeta_p)$ vanishes, then
    $$(\theta_S^T)^{\sharp} \in \Fitt_{\zp G_-}^{\max}((A_{K,T}^-)^{\vee})$$
    whenever $S_p \subset S$.
    \end{theo}

    \begin{cor}
    Let $K/k$ be a Galois CM-extension and $p$ an odd prime. Then $B(K/k,S,p)$ and $BS(K/k,S,p)$ hold whenever $S_p \subset S$ and Iwasawa's
    $\mu$-invariant vanishes.
    \end{cor}

    \begin{proof}
    Since $BS(K/k,S,p)$ implies $B(K/k,S,p)$ by \cite{ich-stark}, Lemma 2.9, we only have to treat the case of the Brumer-Stark conjecture.
    But $BS(K/k,S,p)$ is implied by the so-called strong Brumer-Stark property by \cite{ich-stark}, Prop.~3.8. Recall that this property is fulfilled if
    $\theta_S^T \in \Fitt_{\zp G_-}^{\max}(A_{K,T}^-)$. But in fact the proof of \cite{ich-stark}, Prop.~3.8 carries over unchanged once we observe
    that
    $$\Ann_{\zpg_-}(M) = \Ann_{\zpg_-}(M^{\vee})^{\sharp}$$
    for any finite $\zpg_-$-module $M$.
    \end{proof}

    \begin{proof}[Proof of Theorem \ref{strongBS}]
    Let $\mc K$ be the cyclotomic $\zp$-extension of $K$ and $\mc G = \Gal(\mc K / k)$. There is an isomorphism of $\La(\mc G)_-$-modules
    $$\mc A_{\mc K, \mc T}^- \simeq \varinjlim_n A_{K_n,T_n}^-,$$
    where $T_n := T(K_n)$ and the transition maps in the injective limit are injective (cf.~\cite{EIMC-reform}, Lemma 2.9).
    Hence we may choose a sufficiently large integer $n$ such that $A_{K,T}^- \subset \mc A_{\mc K, \mc T}^-[p^n]$ which induces a natural epimorphism
    \beq \label{epi1}
        \mc A_{\mc K, \mc T}^-[p^n]^{\vee} \onto (A_{K,T}^-)^{\vee}.
    \eeq
    For any $\La(\mc G)$-module $M$ let $M^{\ast} := \Hom_{\zp}(M, \zp)$ endowed with the contravariant $\mc G$-action.
    We have an isomorphism $\mc A[p^m]^{\vee} \simeq T_p(\mc A)^{\ast} \otimes_{\zp} \Z / p^m \Z$ of $\La(\mc G)$-modules
    for any positive integer $m$ and any $\zp$-torsion, divisible
    $\La(\mc G)$-module $\mc A$ of finite local corank (cf.~\cite{EIMC-reform}, Lemma 6.6; the assumption on $\mc G$ being abelian is not
    necessary). This together with (\ref{epi1}) and the dual of sequence (\ref{ses-Iwasawa-motive}) leads to the following epimorphisms
    $$(T_p(\mc M_{\mc S, \mc T}^{\mc K})^-)^{\ast} \onto T_p(\mc A_{\mc K, \mc T}^-)^{\ast} \onto \mc A_{\mc K, \mc T}^-[p^n]^{\vee} \onto (A_{K,T}^-)^{\vee}.$$
    Since $\Ga = \Gal(\mc K / K)$ acts trivially on $(A_{K,T}^-)^{\vee}$, we obtain an epimorphism of $\zp G_-$-modules
    $$(T_p(\mc M_{\mc S, \mc T}^{\mc K})^-)^{\ast}_{\Ga} \onto (A_{K,T}^-)^{\vee}.$$
    But $(\Psi_S^T)^{\sharp}$ is a generator of the Fitting invariant of $(T_p(\mc M_{\mc S, \mc T}^{\mc K})^-)^{\ast}$ by Theorem \ref{EIMC-GP}
    and \cite{ich-Fitting}, Prop.~6.3 (i). Consequently,
    $$(\theta_S^T)^{\sharp} = (\Psi_S^T)^{\sharp}(0) \in \Fitt_{\zpg_-}^{\max}((T_p(\mc M_{\mc S, \mc T}^{\mc K})^-)^{\ast}_{\Ga}) \subset
    \Fitt_{\zpg_-}^{\max}((A_{K,T}^-)^{\vee})$$
    by \cite{ich-Fitting}, Th.~6.4 and Prop.~3.5 (i).
    \end{proof}

    \section{The non-abelian Coates-Sinnott conjecture}

    In this section, we discuss an analogue of the strong Brumer-Stark property for higher \'{e}tale cohomology. We once more
    recall that a Galois extension $K/k$ with Galois group $G$ fulfills this property at an odd prime $p$ if
    \beq \label{strong-BS-property}
        \theta_S^T \in \Fitt_{\zp G}^{\max}(A_{K,T})
    \eeq
    for any two finite sets $S$ and $T$ of places of $k$ such that $Hyp(S,T)$ is satisfied. Note that this does not hold
    in general, even if $G$ is abelian, as follows from the results in \cite{Gr-Kurihara}. We will see that its higher analogue
    should behave much better.\\

    Let $K/k$ be a Galois extension of number fields with Galois group $G$ and $p$ an odd prime.
    We fix an integer $n>1$ and two finite non-empty sets $S$ and $T$ of places of $k$ such that $S$ contains $S\ram \cup S_{\infty}$ and $S \cap T = \emptyset$.
    We also assume that no $p$-adic place of $k$ lies in $T$. For a finite place $w$ of $K$, we write $K(w)$ for the
    residue field of $K$ at $w$. To these data we associate the complex
    $$C_S^T(K/k, \zp(n)) = C_S^T(\zp(n)) := \cone (R\Ga(\fo_{K,S}[\frac{1}{p}], \zp(n)) \to \bigoplus_{w \in T(K)} R\Ga(K(w), \zp(n)))[-1].$$
    We now state the following conjecture.

    \begin{con}[$SCS(K/k,S,T,p,n)$] \label{SCS-conjecture}
    Let the data $(K/k,S,T,p,n)$ be as above. Then
    $$\theta_S^T(1-n) \in \Fitt_{\zp G}^{\max}(H^2(C_S^T(\zp(n)))).$$
    \end{con}

    We will refer to this conjecture as the (non-abelian) strong Coates-Sinnott conjecture. To see the analogy to the strong
    Brumer-Stark property (\ref{strong-BS-property}), we assume that $Hyp(S,T)$ is satisfied and
    look at the associated cohomology sequence in the case $n=1$:
    $$H^1(C_S^T(K/k, \zp(1))) \into E_{S\cup S_p} \otimes \zp \to (\fo_{S \cup S_p} / \fM_T)\mal \to$$
    $$H^2(C_S^T(K/k, \zp(1))) \onto \cl_{S\cup S_p} \otimes \zp.$$
    In fact, this sequence coincides with sequence (\ref{ray_class_sequence}) and we have canonical identifications
    $$H^1(C_S^T(K/k, \zp(1))) \simeq E_{S\cup S_p}^T \otimes \zp,~ H^2(C_S^T(K/k, \zp(1))) \simeq \cl_{S\cup S_p,T,K} \otimes \zp.$$

    We now study Conjecture \ref{SCS-conjecture} in some detail.
    \begin{lem} \label{torsion-free}
    The complex $C_S^T(\zp(n))$ is acyclic outside degrees $1$ and $2$. Moreover, $H^1(C_S^T(\zp(n)))$ is torsion-free and
    $H^2(C_S^T(\zp(n)))$ is finite.
    \end{lem}
    \begin{proof}
    As $p$ is odd, the complex $R\Ga(\fo_{K,S}[\frac{1}{p}], \zp(n))$ is acyclic outside degrees $1$ and $2$. For any $w \in T(K)$, the cohomology
    of $R\Ga(K(w), \zp(n))$ is concentrated in degree $1$. Hence the associated cohomology sequence is
    \beq \label{coh-sequence}
        0 \to H^1(C_S^T(K/k, \zp(n))) \to H^1_{\et}(\fo_{K,S}[\frac{1}{p}], \zp(n)) \to
    \eeq
    $$\bigoplus_{w \in T(K)} H^1_{\et}(K(w), \zp(n)) \to H^2(C_S^T(K/k, \zp(n))) \to H^2_{\et}(\fo_{K,S}[\frac{1}{p}], \zp(n)) \to 0.$$
    We see that $C_S^T(\zp(n))$ is acyclic outside degrees $1$ and $2$ and that $H^2(C_S^T(\zp(n)))$ is finite, as
    $H^2_{\et}(\fo_{K,S}[\frac{1}{p}], \zp(n))$ and $H^1_{\et}(K(w), \zp(n))$, $w \in T(K)$ are. Moreover there are isomorphisms
    $$H^1_{\et}(\fo_{K,S}[\frac{1}{p}], \zp(n))\tor \simeq (\qp / \zp(n))^{G_K},~H^1_{\et}(K(w), \zp(n)) \simeq (\qp / \zp(n))^{G_{K,w}},$$
    where $G_K$ and $G_{K,w}$ denote the absolute Galois group of $K$ and the absolute decomposition group at $w$, respectively; note that the inertia subgroup
    acts trivially on $\qp / \zp(n)$ for all $w \in T(K)$. Since $T$ is not empty, the natural map
    $$(\qp / \zp(n))^{G_K} \to \bigoplus_{w \in T(K)} (\qp / \zp(n))^{G_{K,w}}$$
    is injective, i.e.~$H^1_{\et}(\fo_{K,S}[\frac{1}{p}], \zp(n))\tor$ injects into $\bigoplus_{w \in T(K)} H^1_{\et}(K(w), \zp(n))$.
    This shows that $H^1(C_S^T(\zp(n)))$ is in fact torsion-free.
    \end{proof}

    The relation to the classical Coates-Sinnott conjecture as formulated in \cite{Coates-Sinnott} is the following.
    \begin{prop}
    Assume that $SCS(K/k,S,T,p,n)$ holds. Then
    $$\mc H_p(G) \cdot \theta_S^T(1-n) \subset \Ann_{\zpg}(H^2_{\et}(\fo_{K,S}[\frac{1}{p}], \zp(n))).$$
    In particular, if $G$ is abelian and $SCS(K/k,S,T,p,n)$ holds for all admissible sets $T$, then
    $$\Ann_{\zpg}(H^1_{\et}(\fo_{K,S}[\frac{1}{p}], \zp(n))\tor) \cdot \theta_S(1-n) \subset \Ann_{\zpg}(H^2_{\et}(\fo_{K,S}[\frac{1}{p}], \zp(n))).$$
    In fact, it suffices to consider sets $T$ which only consist of one place.
    \end{prop}

    \begin{proof}
    The surjection $H^2(C_S^T(K/k, \zp(n))) \onto H^2_{\et}(\fo_{K,S}[\frac{1}{p}], \zp(n))$ of sequence (\ref{coh-sequence}) implies an inclusion
    of Fitting invariants
    $$\Fitt_{\zp G}^{\max}(H^2(C_S^T(\zp(n)))) \subset \Fitt_{\zp G}^{\max}(H^2_{\et}(\fo_{K,S}[\frac{1}{p}], \zp(n))).$$
    The first assertion now follows from Theorem \ref{annihilation-theo}. For the second assertion, we observe that, for abelian $G$,
    one has $\mc H_p(G) = \zp G$ and
    $\Ann_{\zpg}(H^1_{\et}(\fo_{K,S}[\frac{1}{p}], \zp(n))\tor)$ is generated by the elements $\de_{\{v\}}(1-n)$, where $v$ runs through the finite places
    of $k$ not above $p$, $v \not\in S\ram$ by a Lemma of Coates \cite{Coates-annihilator}.
    \end{proof}

    \begin{lem} \label{Fitt_de_T}
    The finite $G$-module $\bigoplus_{w \in T(K)} H^1_{\et}(K(w), \zp(n))$ is cohomologically trivial and $\de_T(1-n)$ is a generator of
    $\Fitt_{\zpg}(\bigoplus_{w \in T(K)} H^1_{\et}(K(w), \zp(n)))$.
    \end{lem}
    \begin{proof}
    Let $v \in T$ and $w \in T(K)$ be a place above $v$. The complex $R\Ga(K(w), \zp(n))$ is $\zp G_w$-perfect and its cohomology is concentrated
    in degree $1$. Hence $H^1_{\et}(K(w), \zp(n)))$ is c.t.~as $G_w$-module and hence
    $\bigoplus_{w|v} H^1_{\et}(K(w), \zp(n)) = \ind_{G_w}^G H^1_{\et}(K(w), \zp(n))$ is c.t.~as $G$-module. More precisely, $H^1_{\et}(K(w), \zp(n))$
    is the cokernel of the injective map
    $$\zp G_w \to \zp G_w,~x \mapsto x \cdot (1 - N(v)^n \phi_w\me).$$
    Hence, its Fitting invariant is generated by $\nr(1 - N(v)^n \phi_w\me)$. Since $\de_T(1-n) = \prod_{v \in T} \nr(1 - N(v)^n \phi_w\me)$,
    we are done.
    \end{proof}

    \begin{prop} \label{basechange-prop}
    Let $U$ be a normal subgroup of $G$ and put $F := K^U$ and $\ol G = G/U$. Then we have an isomorphism
    $$\zp \ol G \otimes^L_{\zp G} C_S^T(K/k, \zp(n)) \simeq C_S^T(F/k, \zp(n))$$
    in $\mc D(\zp \ol G)$.
    In particular, $SCS(K/k,S,T,p,n)$ implies $SCS(F/k,S,T,p,n)$.
    \end{prop}

    \begin{proof}
    The first assertion follows, since the corresponding statement holds for the complexes $R\Ga(\fo_{K,S}[\frac{1}{p}], \zp(n))$ and
    $\bigoplus_{w \in T(K)} R\Ga(K(w), \zp(n))$
    (the latter follows easily from Lemma \ref{Fitt_de_T} above). In particular we have canonical isomorphisms
    $$H^1(C_S^T(K/k, \zp(n)))^U  \simeq  H^1(C_S^T(F/k, \zp(n)))$$
    \beq \label{basechange_H}
        H^2(C_S^T(K/k, \zp(n)))_U  \simeq  H^2(C_S^T(F/k, \zp(n))).
    \eeq
    Since $\theta_S^T(K/k, 1-n)$ is mapped to $\theta_S^T(F/k,1-n)$ by the canonical projection $\zp G \onto \zp \ol G$, the natural behavior of
    Fitting invariants implies that if $\theta_S^T(K/k, 1-n)$ lies in $\Fitt_{\zp G}^{\max}(H^2(C_S^T(K/k, \zp(n))))$, then
    $$\theta_S^T(F/k,1-n) \in \Fitt_{\zp \ol G}^{\max}(H^2(C_S^T(K/k, \zp(n)))_U) = \Fitt_{\zp \ol G}^{\max}(H^2(C_S^T(F/k, \zp(n))))$$
    as desired.
    \end{proof}

    \begin{cor}
    It suffices to prove Conjecture \ref{SCS-conjecture} under the additional assumption that $k$ is totally real and $K$ is totally imaginary.
    \end{cor}

    \begin{proof}
    Since $\theta_S^T(1-n) = 0$ if $k$ is not totally real, Conjecture \ref{SCS-conjecture} holds trivially in this case. So we may assume
    that $k$ is totally real. Moreover, Proposition \ref{basechange-prop} implies that we may assume that $K$ is totally imaginary.
    \end{proof}

    We also provide another useful reduction step.
    \begin{lem} \label{reduction-S}
    It suffices to prove Conjecture \ref{SCS-conjecture} under the assumption $S_p \subset S$.
    \end{lem}
    \begin{proof}
    We have an equality
    $$\theta_{S \cup S_p}^T(1-n) = \theta_S^T(1-n) \cdot \prod_{v \in S_p, v \not\in S} \nr(1 - N(v)^{1-n} \phi_w\me),$$
    where $w$ is a place of $K$ above $v$. But the product on the righthand side lies in $\nr(\zpg\mal) = \nr(K_1(\zpg))$
    such that the claim follows by $\nr(\zpg)$-equivalence.
    \end{proof}

    We will henceforth assume that $k$ is totally real, $K$ is totally imaginary and $S_p \subset S$.
    For any $w \in S_{\infty}(K)$, the decomposition group $G_w$ is cyclic of order
    two and we denote its generator by $j_w$.
    Consider the normal subgroup
    $$U := \langle j_w \cdot j_{w'} \mid w, w' \in S_{\infty}(L) \rangle$$
    of $G$. The fixed field $K^{CM} := K^U$ is the maximal CM-subfield of $K$ and is Galois over $k$ with group $\ol G := G/U$.
    We already know by Proposition \ref{basechange-prop} that
    $SCS(K/k,S,T,p,n)$ implies $SCS(K^{CM}/k,S,T,p,n)$. We now
    show that under a mild hypothesis the converse is also true.

    \begin{lem} \label{reduction-CM}
    Assume that $U$ is a $2$-group. Then
    $$SCS(K/k,S,T,p,n) \iff SCS(K^{CM}/k,S,T,p,n).$$
    \end{lem}

    \begin{proof}
    If $U$ is a $2$-group, then the idempotent $\ve_U := |U|\me \sum_{u \in U} u$ lies in $\zpg$. Hence we have decompositions
    $$\zpg  = \ve_U \zpg \op (1 - \ve_U) \zpg,$$
    \bea
        \Fitt_{\zp G}^{\max}(H^2(C_S^T(\zp(n)))) & = & \Fitt_{\ve_U\zp G}^{\max}(\ve_UH^2(C_S^T(\zp(n)))) \op \\
        & & \Fitt_{(1-\ve_U)\zp G}^{\max}((1-\ve_U)H^2(C_S^T(\zp(n)))).
    \eea
    The first term of the latter decomposition naturally identifies with $\Fitt_{\zp \ol G}^{\max}(H^2(C_S^T(\zp(n)))_U)$
    and the result follows from (\ref{basechange_H}) once we observe that $\theta_S^T(K/k, 1-n) = \ve_U \cdot \theta_S^T(K/k, 1-n)$
    maps to $\theta_S^T(K^{CM}/k,1-n)$ under the natural identification $\ve_U \qpg \simeq \qp \ol G$.
    \end{proof}

    \begin{rem}
    Note that $U$ is a $2$-group if $G$ has a unique $2$-Sylow subgroup. This in particular applies to nilpotent groups.
    Moreover, if our primary interest was in an extension $K/k$ of totally real fields, then we may enlarge $K$ to
    a CM-field.
    \end{rem}

    We will now focus on CM-extensions $K/k$ with Galois group $G$. We denote the maximal real subfield of $K$ by $K^+$ and
    let $j \in G$ be complex conjugation. We put
    $e_n := \frac{1 + (-1)^n j}{2}$ which is a central idempotent in $G$. For any $\zpg$-module $M$ we have
    natural isomorphisms
    $$e_n \cdot M = \left\{ \barr{lll} M^+ & \mbox{ if } & n \mbox{ is even}\\ M^- & \mbox{ if } & n \mbox{ is odd}.\earr \right.$$
    Since $e_n \theta_S^T(1-n) = \theta_S^T(1-n)$, Conjecture \ref{SCS-conjecture} is true if and only if $\theta_S^T(1-n)$ belongs
    to $\Fitt_{e_n \zp G}^{\max}(e_n H^2(C_S^T(\zp(n))))$. Let
    $$C_{p,r}\punkt := R\Hom_{\zp}(R\Ga_c(\fo_{K,S}[1/p], \zp(1-n)), \zp[-2]),$$
    where $R\Ga_c(\fo_{K,S}[1/p], \zp(1-n))$ is the complex of $\zp G$-modules given by the cohomology with compact support as defined in \cite{Burns_Flach},
    p.~522.
    Then the complex $C_{p,r}\punkt$ belongs to $\mc D^{\perf}(\zpg)$ and fits into an exact triangle
    in $\mc D(\zp G)$ (cf.~\cite{Burns_Flach_invariants}, Prop.~4.1).:
    $$
        \bigoplus_{w \in S_{\infty}(K)} R\Hom_{\zp}(R\Ga_{\De}(K(w), \zp(1-n)), \zp)[-3] \lto
        R\Ga(\fo_{K,S}[1/p], \zp(n)) \lto C_{p,r}\punkt[-1] \lto
    $$
    where $R\Hom_{\zp}(R\Ga_{\De}(K(w), \zp(1-n)), \zp)$ is given by $\zp(n-1)$ (placed in degree zero) if $w$ is complex, and by
    $$\zp \stackrel{\de_1}{\lto} \zp \stackrel{\de_0}{\lto} \zp \stackrel{\de_1}{\lto} \dots$$
    if $w$ is real (which does not occur here) and $\de_i$ is multiplication with $1- (-1)^{i-n}$ for $i=0,1$; here, the first $\zp$ is placed in degree $0$.
    Hence the only non-trivial term of $R\Hom_{\zp}(R\Ga_{\De}(L(w), \zp(1-n)), \zp)$ is $\bigoplus_{w \in S_{\infty}(K)} \zp(n-1)$ which is
    annihilated by $e_n$. Hence we have an isomorphism
    $$e_n R\Ga(\fo_{K,S}[1/p], \zp(n)) := e_n \zpg \otimes^L_{\zpg} R\Ga(\fo_{K,S}[1/p], \zp(n)) \simeq e_n C_{p,r}\punkt[-1]$$
    in $\mc D^{\perf}(e_n \zpg)$. In fact
    \beq \label{en-is-tor}
        e_n  H^1_{\et}(\fo_{K,S}[1/p], \zp(n)) = H^1_{\et}(\fo_{K,S}[1/p], \zp(n))\tor
    \eeq
    such that we obtain an exact triangle
    \beq \label{BF-triangle}
        e_n C_S^T(\zp(n))  \to e_n C_{p,r}\punkt[-1]  \to e_n \bigoplus_{w \in T(K)} R\Ga(K(w), \zp(n))  \to
    \eeq
    in $\mc D^{\perf}\tor(e_n \zpg)$. The following theorem gives the relation to the ETNC
    as formulated by Burns and Flach \cite{Burns_Flach}.

    \begin{theo} \label{connection-ETNC}
    Assume that $K/k$ is a Galois CM-extension. Then $e_n H^1(C_S^T(\zp(n)))$ vanishes and $e_n H^2(C_S^T(\zp(n)))$
    is a cohomologically trivial $G$-module. Moreover, the following assertions are equivalent.
    \ben
        \item
        $\theta_S^T(1-n) \in \Fitt_{e_n \zpg}(e_n H^2(C_S^T(\zp(n))))$.
        \item
        $\theta_S^T(1-n)$ is a generator of $\Fitt_{e_n \zpg}(e_n H^2(C_S^T(\zp(n))))$.
        \item
        The $p$-part of the ETNC for the pair $(\Q(1-n)_K, e_n \Z[\half] G)$ holds.
    \een
    In particular, (1) and (2) are independent of the sets $S$ and $T$.
    \end{theo}

    \begin{cor} \label{ETNC-cor}
    Assume that $K/k$ is a Galois extension of number fields with $k$ totally real and let $p$ be an odd prime.
    If there exists a totally imaginary field $\ti K$
    containing $K$ such that $\ti K /k$ is Galois and $\ti K / \ti K^{CM}$ is a $2$-extension, then $SCS(K/k,S,T,p,n)$
    holds for all admissible sets $S$ and $T$ provided that Iwasawa's $\mu$-invariant attached to the cyclotomic $\zp$-extension
    of $\ti K^{CM}(\zeta_p)$ vanishes.
    \end{cor}

    \begin{proof}
    By Proposition \ref{basechange-prop} and Lemma \ref{reduction-CM} we are reduced to the case $K = \ti K^{CM}$, i.e.~$K$
    is actually a CM-field. Now \cite{Burns-mc}, Cor.~2.10 shows that the relevant part of the ETNC holds if $\mu =0$.
    \end{proof}

    \begin{rem}
    We will sketch a second proof of Corollary \ref{ETNC-cor} below, using our results of section 2 and 3.
    \end{rem}

    The following is a non-abelian analogue of \cite{EIMC-reform}, Th.~6.11 and also reproves \cite{ich-negative}, Cor.~4.2.
    \begin{cor}
    Assume that $K/k$ is a Galois CM-extension and let $p$ be an odd prime. If Iwasawa's $\mu$-invariant attached to the cyclotomic $\zp$-extension of $K(\zeta_p)$ vanishes, then
    $$\Fitt_{\zp G}^{\max}(H^1_{\et}(\fo_{K,S}[1/p], \zp(n))\tor^{\vee})^{\sharp} \cdot \theta_S(1-n) = e_n \Fitt_{\zp G}^{\max}(H^2_{\et}(\fo_{K,S}[1/p], \zp(n))).$$
    \end{cor}

    \begin{proof}
    We will make use of the validity of the above special case of the ETNC under the assumption $\mu=0$.
    The $e_n$-part of the exact sequence (\ref{coh-sequence}) is a four term sequence of finite $e_n \zpg$-modules. The two middle terms are c.t.~with
    generator $\de_T(1-n)$ and $\theta_S^T(1-n)$ by Lemma \ref{Fitt_de_T} and Theorem \ref{connection-ETNC}, respectively.
    The result now follows by applying \cite{ich-Fitting}, Prop.~5.3 (ii) to this four term sequence.
    \end{proof}

    \begin{rem}
    Since $H^1_{\et}(\fo_{K,S}[1/p], \zp(n))\tor$ is a cyclic module, we always have an inclusion
    $$\nr(\Ann_{\zpg} (H^1_{\et}(\fo_{K,S}[1/p], \zp(n))\tor)) \subset \Fitt_{\zp G}^{\max}(H^1_{\et}(\fo_{K,S}[1/p], \zp(n))\tor^{\vee})^{\sharp}$$
    with equality if $G$ is abelian (and one can drop $\nr$ as it is just the identity on $\zpg$ for abelian $G$). For arbitrary $G$, equality seems to be likely, but is not clear.
    \end{rem}

    \begin{proof}[Proof of Theorem \ref{connection-ETNC}]
    Since $e_n H^1(C_S^T(\zp(n)))$ is torsion-free by Lemma \ref{torsion-free} and is a submodule of
    $e_n H^1_{\et}(\fo_{K,S}[1/p], \zp(n))$ which is the torsion submodule of
    $H^1_{\et}(\fo_{K,S}[1/p], \zp(n))$ by (\ref{en-is-tor}), it must by trivial. The triangle (\ref{BF-triangle}) shows
    that $C_S^T(\zp(n)) e_n$ is a perfect complex. But the only non-trivial cohomology group is $e_n H^2(C_S^T(\zp(n)))$
    which is thus of finite projective dimension, hence c.t.~as $G$-module. In particular, it admits a quadratic presentation
    and, using the exact triangle (\ref{BF-triangle}) and Lemma \ref{Fitt_de_T},
    \bea
        \Fitt_{e_n \zpg}(e_n H^2(C_S^T(\zp(n)))) & = & \Fitt_{e_n \zpg}(e_n C_S^T(\zp(n)) )\me\\
        & = & \Fitt_{e_n \zpg}(e_n \bigoplus_{w \in T(K)} R\Ga(K(w), \zp(n))) \\
        & & \cdot \Fitt_{e_n \zpg}(e_n C_{p,r}\punkt[-1])\me\\
        & = & \de_T(1-n) \cdot \Fitt_{e_n \zpg}(e_n C_{p,r}\punkt).
    \eea
    But $\Fitt_{e_n \zpg}(e_n C_{p,r}\punkt)$ is generated by $\theta_S(1-n)$ if and only if the refined Euler characteristic
    $\chi_{e_n \zpg, e_n \qpg}(e_n C_{p,r}\punkt,0)$ equals $\hat\partial_G(\theta_S(1-n))$, i.e.~if and only if the ETNC for
    the pair $(\Q(1-n)_K, e_n \Z[\half] G)$ holds; this reformulation of the ETNC is due to Burns \cite{Burns-motivic}, Prop.~4.2.6,
    but see \cite{ich-negative}, Prop.~2.15 which applies more directly. This shows the equivalence of (2) and (3). Since clearly
    (2) implies (1), we are left with the proof of $(1) \implies (2)$. For this, let $E$ be a splitting field of $\qp G$; then
    $$E \otimes \zeta(\qp G) = \zeta(EG) = \bigoplus_{\chi \in \irrp(G)} E e_{\chi}$$
    and we may write $1 \otimes \theta_S^T(1-n) = \sum_{\chi} \theta_S^T(1-n)_{\chi} e_{\chi}$. By \cite{ich-Fitting}, Prop.~5.4 it suffices to show
    that
    $$\prod_{\chi} \theta_S^T(1-n)_{\chi}^{\chi(1)} \sim |e_n H^2(C_S^T(\zp(n)))|,$$
    where the product runs through all irreducible odd (resp.~even) characters of $G$ if $n$ is odd (resp.~even) and $\sim$ means ``equal up to a
    $p$-adic unit''. By the same proposition
    and Lemma \ref{Fitt_de_T}, we have
    $$\prod_{\chi} \de_T(1-n,\chi)^{\chi(1)} \sim |e_n \bigoplus_{w \in T(K)} H^1_{\et}(K(w), \zp(n))|.$$
    Using the fact that $e_n H^1(C_S^T(\zp(n)))$ vanishes and sequence (\ref{coh-sequence}), we are left to show (in obvious notation)
    \beq \label{lefttoshow}
        \prod_{\chi} \theta_S(1-n)_{\chi}^{\chi(1)} \sim \frac{|e_n H^2_{\et}(\fo_{K,S}[1/p], \zp(n))|}{|e_n H^1_{\et}(\fo_{K,S}[1/p], \zp(n))|}.
    \eeq
    If $n$ is even, the left hand side equals $\zeta_{K^+}(1-n)$, where $\zeta_{K^+}$ denotes the Dedekind zeta function of the number field $K^+$.
    Moreover,
    $$\frac{|e_n H^2_{\et}(\fo_{K,S}[1/p], \zp(n))|}{|e_n H^1_{\et}(\fo_{K,S}[1/p], \zp(n))|} =
    \frac{|H^2_{\et}(\fo_{K^+,S}[1/p], \zp(n))|}{|H^1_{\et}(\fo_{K^+,S}[1/p], \zp(n))\tor|}$$
    such that (\ref{lefttoshow}) is equivalent to the cohomological version of Lichtenbaum's conjecture which is a theorem due to Wiles \cite{Wiles-main}
    in this case.
    If $n$ is odd, a similar argument shows that (\ref{lefttoshow}) is equivalent to a higher relative class number formula as formulated and proved
    by Kolster \cite{Kolster_relativecnf}, Prop.~1.1.
    \end{proof}

    Finally, we briefly illustrate how to use the results of section 2 and 3 to give an alternative proof of Corollary \ref{ETNC-cor}.
    Since this will not lead to a new result, some of the details are left to the reader. As before,
    we may reduce the problem to the case, where $K/k$ is a CM-extension.
    In fact by Theorem \ref{connection-ETNC}, the following provides a new proof of the $p$-part of the ETNC for the pair $(\Q(1-n)_K, e_n \Z[\half] G)$
    if $\mu=0$.\\

    By Lemma \ref{reduction-S}, we may assume that $S$ contains
    the $p$-adic places, and by Proposition \ref{basechange-prop} we may assume that $\zeta_p \in K$. Let $\mc K$ be the cyclotomic
    $\zp$-extension of $K$ and denote by $\mc S$ and $\mc T$ the places of $\mc K$ above the places in $S$ and $T$, respectively.
    As before, let $\mc G = \Gal(\mc K/k)$ and $\Ga = \Gal(\mc K / K)$.
    The exact sequence (\ref{GP-sequence}) tensored with $\zp(n-1)$ leads to an exact sequence of $\La(\mc G)$-modules
    $$\zp(n) \into T_p(\De_{\mc K, \mc T})^-(n-1) \to T_p(\mc M_{\mc S, \mc T}^{\mc K})^-(n-1) \onto X_S^+(-n)^{\ast}.$$
    A spectral sequence argument leads to natural isomorphisms of $\zpg$-modules (cf.~\cite{EIMC-reform}, Prop.~6.17; the
    assumption on $G$ to be abelian is not necessary)
    $$e_n H^2_{\et}(\fo_{K,S}[1/p], \zp(n)) \simeq X_S^+(-n)^{\ast}_{\Ga},~e_n H^1_{\et}(\fo_{K,S}[1/p], \zp(n))\tor \simeq \zp(n)_{\Ga}.$$
    Since also
    $$T_p(\De_{\mc K, \mc T})^-(n-1)_{\Ga} \simeq e_n \bigoplus_{w \in T(K)} H^1_{\et}(K(w), \zp(n)),$$
    taking $\Ga$-coinvariants  yields an exact sequence
    $$H^1_{\et}(\fo_{K,S}[1/p], \zp(n))\tor \into e_n \bigoplus_{w \in T(K)} H^1_{\et}(K(w), \zp(n)) \to$$
    $$T_p(\mc M_{\mc S, \mc T}^{\mc K})^-(n-1)_{\Ga} \onto e_n H^2_{\et}(\fo_{K,S}[1/p], \zp(n))$$
    which is rather similar to sequence (\ref{coh-sequence}) times $e_n$. In fact, if the extension class of $\al(\_)(1)$ applied to
    sequence (\ref{GP-sequence}) matches the extension class of the Ritter-Weiss sequence (\ref{canonical_complex}), i.e.~if
    the complex which consists of the two middle terms is isomorphic to $R\Hom(R\Ga_{\et}(\Spec(\fo_{\mc K}[\frac{1}{S}]), \qp / \zp), \qp / \zp)$ in $\mc D(\La(\mc G))$,
    then $T_p(\mc M_{\mc S, \mc T}^{\mc K})^-(n-1)_{\Ga}$ naturally identifies with $e_n H^2(C_S^T(\zp(n)))$. If not, one can construct
    a four term exact sequence
    $$\zp(n) \into T_p(\De_{\mc K, \mc T})^-(n-1) \to Y_S^T(n-1) \onto X_S^+(-n)^{\ast}$$
    which has the correct extension class, and a proof similar to that of Theorem \ref{EIMC-GP} shows that
    $\Psi_{S,T}$ is a generator of $\Fitt_{\La(\mc G)_-}(Y_S^T)$.
    In any case, we have an equality of Fitting invariants
    $$\Fitt_{e_n \zpg}(T_p(\mc M_{\mc S, \mc T}^{\mc K})^-(n-1)_{\Ga}) = \Fitt_{e_n \zpg}(e_n H^2(C_S^T(\zp(n)))).$$
    But Theorem \ref{EIMC-GP} implies that the left hand side is generated by $t_{n-1}(\Psi_S^T)(0) = \theta_S^T(1-n)$ as desired;
    here, for $m \in \Z$ we denote by $t_{m}$ the continuous $\zp$-algebra endomorphism of $\La(\mc G)$ induced by
    $t_m(g) = \ka(g)^m \cdot g$ and we have used the following fact.
    Let $M$ be a finitely generated torsion $\La(\mc G)$-module of projective dimension at most $1$
    which has no non-trivial finite submodule.
    Then $t_m(\Psi)$ is a generator of
    the Fitting invariant of $M(m)$ if $\Psi$ is a generator of the Fitting invariant of $M$;
    this follows from the proof of Proposition \ref{Fitt-alpha}.

\noindent Andreas Nickel~~ anickel3@math.uni-bielefeld.de\\
Universit\"{a}t Bielefeld,
    Fakult\"{a}t f\"{u}r Mathematik,
    Postfach 100131,
    33501 Bielefeld,
    Germany

\end{document}